\documentclass[leqno,11pt]{amsart}

\usepackage{amsmath,amstext,amssymb,amsopn,amsthm,mathrsfs}
\usepackage{verbatim}

\title[Operators related to symmetrized Jacobi expansions]
{Harmonic analysis operators related to symmetrized Jacobi expansions for all admissible parameters}

\author[B{.} Langowski]{Bartosz Langowski}
\address{Bartosz Langowski \newline
			Faculty of Pure and Applied Mathematics, 
      Wroc\l{}aw University of Technology       \newline
      Wyb{.} Wyspia\'nskiego 27,
      50--370 Wroc\l{}aw, Poland      
      }
\email{bartosz.langowski@pwr.edu.pl}

\allowdisplaybreaks

\theoremstyle{plain}
\newtheorem{thm}{Theorem}[section]

\newtheorem{lem}[thm]{Lemma}
\newtheorem{prop}[thm]{Proposition}
\newtheorem{cor}[thm]{Corollary}

\theoremstyle{definition}

\theoremstyle{remark}
\newtheorem*{rem*}{Remark}

\theoremstyle{plain}

\DeclareMathOperator{\domain}{Dom}

\DeclareMathOperator{\sign}{sign}

\def\ko{\delta_N^{\textrm{odd}}}
\def\R{\mathbb R}
\def\N{\mathbb N}
\def\Z{\mathbb Z}

\def\d{d}

\def\P{\mathcal P}

\def\m{\mu} 						
\def\ab{\alpha,\beta}
\def\J{\mathcal J} 			
\def\q{\mathfrak q}
\def\de{\delta_N^{\textrm{even}}}

\def\P{\mathcal P}

\def\m{\mu} 						
\def\J{\mathcal J}

\def\a{\alpha}
\def\b{\beta}

\def\t{\theta}
\def\vp{\varphi}
\def\st{\sin \frac{\theta}{2}}
\def\svp{\sin \frac{\varphi}{2}}
\def\ct{\cos \frac{\theta}{2}}
\def\cvp{\cos \frac{\varphi}{2}}

\def\pia{d\Pi_{\alpha}(u)}
\def\pib{d\Pi_{\beta}(v)}

\def\q{\mathfrak{q}}

\def\piK{d\Pi_{\alpha, K}(u)}
\def\piR{d\Pi_{\beta, R}(v)}

\begin{document}

\begin{abstract}
This is an ultimate completion of our earlier paper [Acta.\ Math.\ Hungar.\ 140 (2013), 248--292]
where mapping properties of several fundamental harmonic analysis operators in the setting of symmetrized
Jacobi trigonometric expansions were investigated under certain restrictions on the underlying parameters
of type. In the present article we take advantage of very recent results due to Nowak, Sj\"ogren and
Szarek to fully release those restrictions, and also to provide shorter and more transparent proofs
of the previous restricted results. Moreover, we also study mapping properties of analogous operators
in the parallel context of symmetrized Jacobi function expansions. Furthermore, as a consequence of
our main results we conclude some new results related to the classical non-symmetrized Jacobi polynomial
and function expansions.
\end{abstract}

\maketitle

\footnotetext{
\emph{\noindent 2010 Mathematics Subject Classification:} primary 42C10; secondary 42C05, 42C20\\
\emph{Key words and phrases:} Jacobi expansion, Jacobi operator, symmetrization, Poisson semigroup, maximal operator, Riesz transform, square function, spectral multiplier, Calder\'on-Zygmund operator.
} 
\section{Introduction} \label{sec:intro}

In \cite{NoSt} Nowak and Stempak postulated a unified conjugacy scheme in the context of general orthogonal
expansions related to a second order differential operator, a `Laplacian'. Later the same authors
in \cite{Symmetrized} proposed a symmetrization procedure pertaining to the theory of \cite{NoSt}
that allowed them to overcome the lack of symmetry in a decomposition of the related Laplacian and,
consequently, eliminate substantial deviations from the classical theory.
It was shown in \cite{Symmetrized} that the symmetrization is reasonable as far as $L^2$ theory is concerned.
However, the question of validity from the $L^p$ theory perspective was left open, being practically
impossible to be answered on the assumed level of generality. Thus a natural problem arose, namely
to test $L^p$ flavor of the symmetrization in selected concrete, possibly classical contexts.

To a large extent, this motivated our paper \cite{L1} where the symmetrization was applied in the framework
of Jacobi trigonometric polynomial expansions. More precisely, we proved that fundamental harmonic analysis
operators in the Jacobi symmetrized setting, including Riesz transforms, Littlewood-Paley-Stein type square
functions, Jacobi-Poisson semigroup maximal operator and certain spectral multipliers,
are bounded on weighted $L^p$ spaces and are of weighted weak type $(1,1)$. Analogous results in the original
non-symmetrized Jacobi context were obtained earlier by Nowak and Sj\"ogren \cite{NS1}, by means
of the Calder\'on-Zygmund theory. The results of \cite{L1} heavily depend on the techniques developed in
\cite{NS1} and, consequently, inherit the restriction $\ab \ge -1/2$ on the Jacobi parameters of type.
The reason of this restriction in \cite{NS1} was the lack of suitable integral representation of the
Jacobi-Poisson kernel for other values of $\a$ and $\b$.

The latter obstacle was recently overcome by Nowak, Sj\"ogren and Szarek \cite{parameters}. The authors
established an integral formula for the Jacobi-Poisson kernel valid for all admissible $\ab > -1$.
This new general formula is essentially more complicated comparing to the one from \cite{NS1}
(actually, it contains the one from \cite{NS1} as a special case), nevertheless it turned out to be suitable
for developing methods parallel to those from \cite{NS1} and then proving weighted $L^p$ boundedness of
several fundamental harmonic analysis operators. In fact, \cite{parameters} offers some new ideas
and improvements which, in particular, simplify and clarify the earlier analysis in the restricted
case $\ab \ge -1/2$.

The main purpose of the present paper is to take advantage of \cite{parameters} in order to remove
the restriction on $\a$ and $\b$ imposed in \cite{L1}, see Theorem \ref{thm:main}, and also to provide
shorter and more transparent proofs of the main results from \cite{L1}.
All this completes the research undertaken in \cite{L1}.

Another aim of this work is to examine an alternative Jacobi symmetrized setting, this time originating from
the discrete system of Jacobi trigonometric functions rather than polynomials. Sobolev and potential
spaces in this symmetrized context were investigated recently by the author in \cite{L4}. Here we focus
on weighted $L^p$ mapping properties of the fundamental harmonic analysis operators mentioned above. 
It occurs that weighted $L^p$ boundedness of these operators, see Theorem \ref{thm:mainf}, can be
concluded in a rather straightforward manner from the analogous results in the symmetrized Jacobi
polynomial setting.

We point out that analysis in Jacobi settings received a considerable attention in recent years, see
for instance \cite{ABCSS,CNS,L1,L2,L3,L4,Nowak&Roncal,NS1,NS2,parameters,Stempak} 
and numerous other references given in these articles. In particular, mapping properties of harmonic analysis
operators in the classical (non-symmetrized) discrete Jacobi polynomial and function contexts were
extensively studied. Our present results in the symmetrized Jacobi 
settings also contribute to the latter line of research
since they imply new mapping properties in the non-symmetrized Jacobi 
contexts, see Propositions \ref{prop:genfg}
and \ref{prop:genJf} and the accompanying comments in Section \ref{ssec:33}.

The paper is organized as follows. The remaining part of this section describes the notation used
throughout the paper. In Section \ref{sec:prel} we introduce the Jacobi settings to be investigated and
the associated basic notions. Section \ref{sec:thm} contains statements of the main results. These are
either immediately justified or their proofs are reduced to certain kernel estimates. The latter are
technically involved and their proofs require some preparation that is done in Section \ref{sec:tech}.
After that, in Section \ref{sec:ker}, we finally prove the kernel estimates.

\subsection*{Notation}

We always assume that $\ab > -1$, unless explicitly stated otherwise. Let
$$
\Psi^{\ab}(\theta) := \bigg|\sin\frac{\theta}2\bigg|^{\a+1/2}\bigg(\cos\frac{\theta}2\bigg)^{\b+1/2}.
$$
Throughout the paper we use fairly standard notation, with all symbols referring either to the 
metric measure spaces $((-\pi,\pi),\mu_{\ab},|\cdot|)$ and $((-\pi,\pi),d\theta,|\cdot|)$
or to their subspaces
$((0,\pi),\mu_{\ab}^{+},{|\cdot|})$ and $((0,\pi),d\theta, |\cdot|)$, depending on the context. 
Here the measure $\mu_{\ab}$ on the interval $(-\pi,\pi)$ is given by
$$
d\mu_{\ab}(\theta) = \big[ \Psi^{\ab}(\theta) \big]^2 \, d\theta,
$$
and $\mu_{\ab}^{+}$ is the restriction of $\mu_{\ab}$ to the interval $(0,\pi)$.
All the measures appearing above are doubling, hence the four metric measure spaces are actually spaces
of homogeneous type. 

For any function $f$ on $(-\pi,\pi)$, by $f^+$ we mean its restriction to $(0,\pi)$. Further, $\check{f}$
denotes the reflection of $f$, i.e.\ $\check{f}(\theta) = f(-\theta)$, and 
$f_{\textrm{even}}$ and $f_{\textrm{odd}}$ the even and odd parts of $f$, respectively,
$$
f_{\textrm{even}}=\frac{1}2 (f+\check{f}),\qquad f_{\textrm{odd}}=\frac{1}2 (f-\check{f}).
$$
By $\langle f,g\rangle_{d\m_{\ab}}$ we mean 
$\int_{-\pi}^{\pi} f(\theta)\overline{g(\theta)}\, d\m_{\ab}(\theta)$ whenever the integral makes sense,
and similarly for $\langle f,g\rangle_{d\m_{\ab}^{+}}$. Analogously, $\langle f,g\rangle$ stands for the
$L^2$ scalar product with respect to the Lebesgue measure in $(-\pi,\pi)$ or $(0,\pi)$, depending on the
context. 
Weighted $L^p$ spaces with respect to $\mu_{\ab}$ or $\mu_{\ab}^+$ will be written as
$L^p(wd\mu_{\ab})$ and $L^p(wd\mu_{\ab}^{+})$, respectively, $w$ being a nonnegative weight.
We simply write $L^p(w)$ when the underlying measure is the Lebesgue measure in $(-\pi,\pi)$ or $(0,\pi)$.

The Muckenhoupt classes of $A_p$ weights related to the measures $\m_{\ab}$ and $\mu_{\ab}^{+}$
will be denoted by $A_p^{\ab}$ and $(A_p^{\ab})^+$, respectively (see e.g.\ \cite[Section 1]{L1} for the
definitions). Note that a double power even weight
$$
w_{r,s} := \Psi^{r-1/2,s-1/2}
$$
belongs to $A_p^{\ab}$ for a fixed $1\le p < \infty$ if and only if 
$-(2\alpha+2)<r<(2\alpha+2)(p-1)$ and $-(2\beta+2) < s < (2\beta+2)(p-1)$, with the upper inequalities
weakened in case $p=1$.
For our purposes, we also define classes of double power even weights $B_p^{\ab}$, $1 \le p < \infty$,
by the requirement that $w_{r,s} \in B_p^{\ab}$ if and only if
$w_{r+(\a+1/2)(p-2),s+(\b+1/2)(p-2)} \in A_p^{\ab}$. Thus
\begin{align*}
B_p^{\ab}:=\big\{w_{r,s}:  &-1-(\alpha+1/2)p<r< p-1 + (\alpha+1/2)p\\ 
& \textrm{and}\; -1-(\beta+1/2)p<s<p-1 + (\beta+1/2)p\big\}
\end{align*}
with the upper inequalities weakened in case $p=1$. Observe that $B_p^{\ab} \neq \emptyset$ and the
trivial weight $w_{0,0}\equiv 1$ belongs to $B_{p}^{\ab}$ if and only if $\ab \ge -1/2$ or
$\min(\ab) < -1/2$ and $p$ is restricted by the condition $-\min(\ab)-1/2 < 1/p < \min(\ab)+3/2$.

While writing estimates, we will frequently use the notation $X \lesssim Y$
to indicate that $X \le CY$ with a positive constant $C$ independent of significant quantities.
We shall write $X \simeq Y$ when simultaneously $X \lesssim Y$ and $Y \lesssim X$.

\subsection*{Acknowledgment}
The author would like to express his gratitude to Professor Adam Nowak for indicating the topic and
his support during the preparation of this paper.

\section{Preliminaries}\label{sec:prel}

In this section we briefly describe the Jacobi settings we investigate. For the facts we present and also
for further details concerning the non-symmetrized Jacobi contexts the reader is referred, for instance,
to \cite{NS1} and \cite{L2,Stempak}.
In case of the symmetrized settings we refer to the author's papers \cite{L1,L4}.

\subsection{Non-symmetrized Jacobi settings} \label{ssec:21}
Let $P_n^{\ab}$ denote the classical Jacobi polynomials.
It is natural and convenient to apply the trigonometric parametrization 
$x=\cos\theta$ and consider the normalized trigonometric polynomials
\begin{equation} \label{P}
\P_n^{\ab}(\theta) = c_n^{\ab} P_n^{\ab}(\cos\theta),
\end{equation}
where $c_n^{\ab}>0$ are suitable normalizing constants.
It is well known that the system $\{\P_n^{\ab}:n\ge 0\}$ is an orthonormal basis in $L^2(d\m_{\ab}^{+})$. 
Moreover, $\P_n^{\ab}$ are eigenfunctions of the Jacobi differential operator 
\begin{equation}\label{jacobi}
\J_{\ab} = - \frac{d^2}{d\theta^2} - \frac{\alpha-\beta+(\alpha+\beta+1)\cos\theta}{\sin \theta}
	\frac{d}{d\theta} + ( \lambda_0^{\ab})^2,
\end{equation}
being
$$
\J_{\ab} \P_n^{\ab} = \lambda_n^{\ab} \P_n^{\ab}, 
	\qquad n \ge 0,
$$
where the eigenvalues are given by
$$
\lambda_n^{\ab}=\Big( n+ \frac{\alpha+\beta+1}{2}\Big)^2, \qquad n \ge 0.
$$
The decomposition
$$
\J_{\ab} = \delta_{\ab}^*\delta + \lambda_0^{\ab},
$$
determines a natural derivative $\delta$ associated with $\J_{\ab}.$ Here
$$
\delta = \frac{d}{d\theta}, \qquad 
\delta_{\ab}^* = -\frac{d}{d\theta} - \Big(\alpha+\frac{1}2\Big)\cot\frac{\theta}2
+\Big(\beta+\frac{1}{2}\Big)\tan\frac{\theta}2
$$
and $\delta_{\ab}^*$ is the formal adjoint of $\delta$ in $L^2(d\m_{\ab}^{+})$.
Notice that $\delta\neq-\delta_{\ab}^*$ in general.

The system of Jacobi functions arises by adjusting the system of Jacobi trigonometric polynomials so that
the orthogonality measure is the Lebesgue measure $d\t$ in $(0,\pi)$. Thus the Jacobi functions are given by
\begin{equation} \label{Fi}
\phi_n^{\ab} = \Psi^{\ab} \P_n^{\ab}, \qquad n \ge 0,
\end{equation}
and the resulting system $\{\phi_n^{\ab}:n\ge 0\}$ constitutes an orthonormal basis in $L^2(d\t)$.
Moreover, each $\phi_n^{\ab}$ is an eigenfunction of another Jacobi differential operator
\begin{equation*}
L_{\ab}= -\frac{d^2}{d\theta^2}
-\frac{1-4\alpha^2}{16\sin^2\frac{\theta}{2}}-\frac{1-4\beta^2}{16\cos^2\frac{\theta}{2}},
\end{equation*}
the corresponding eigenvalue being again $\lambda_n^{\ab}$. We note that $L_{\ab}$ admits the decomposition
\begin{equation*}
L_{\ab}= D_{\ab}^*D_{\ab}+\lambda_0^{\ab},
\end{equation*}
where
\begin{equation} \label{der_ini}
D_{\ab}=\frac{d}{d\theta}-\frac{2\alpha+1}{4}\cot\frac{\theta}{2}+\frac{2\beta+1}{4}\tan\frac{\theta}{2},
\qquad D_{\ab}^* = D_{\ab}-2\frac{d}{d\theta},
\end{equation}
are the first order derivative naturally associated with $L_{\ab}$ and 
its formal adjoint in $L^2(d\t)$, respectively.

Observe that the Laplacians in the two Jacobi frameworks are conjugated via the identity
$L_{\ab}(\Psi^{\ab}f) = \Psi^{\ab}\mathcal{J}_{\ab}f$, for suitable $f$. Similar relations hold for
the derivatives and many other operators emerging from $\mathcal{J}_{\ab}$ and $L_{\ab}$.

\subsection{Symmetrized Jacobi polynomial setting}\label{poly}

This framework is related to the larger interval$^{\ddag}$\!\!
\footnote{\noindent $\ddag$ Formally, the space
is the union $(-\pi,0) \cup (0,\pi)$. However, usually we will identify it with the
interval $(-\pi,\pi)$, since for the aspects of the theory we are interested in,
such as $L^p$ inequalities,  the single point $\theta=0$ is negligible.
This remark concerns also the symmetrized function setting discussed in the next subsection.}
$(-\pi,\pi)$ and emerges from applying the symmetrization procedure proposed in \cite{Symmetrized}
to the system of Jacobi trigonometric polynomials.
The symmetrized Jacobi operator $\mathbb{J}_{\ab}$ is given by 
$$
\mathbb{J}_{\ab}f = \mathcal{J}_{\ab}f + 
	\frac{(\alpha+\beta+1)+ (\alpha-\beta)\cos\theta}{\sin^2\theta} f_{\textrm{odd}},
$$
where $\mathcal{J}_{\ab}$ is naturally extended to $(-\pi,\pi)$ by \eqref{jacobi}.
This operator can be decomposed as
$$
\mathbb{J}_{\ab}=-\mathbb{D}_{\ab}^2+\lambda_{0}^{\ab},
$$
with the derivative $\mathbb{D}_{\ab}$ given by
$$
\mathbb{D}_{\ab}f=\frac{df}{d\theta}
	+\frac{\alpha-\beta+(\alpha+\beta+1)\cos\theta}{\sin\theta}f_{\textrm{odd}}.
$$
It is remarkable that $\mathbb{D}_{\ab}$ is formally skew-adjoint in $L^2(d\m_{\ab})$.

The orthonormal and complete in $L^2(d\mu_{\ab})$ system $\{\Phi_n^{\ab} : n \ge 0\}$
associated with $\mathbb{J}_{\ab}$ is defined as
\begin{align} \label{Phi}
\Phi_{n}^{\ab}(\t) =\frac{1}{\sqrt{2}} 
\begin{cases}
\P_{n/2}^{\ab}(\t), & n \;\; \textrm{even}, \\
\frac{1}{2}\sin\theta\,\P_{(n-1)/2}^{\alpha+1, \beta+1}(\t), & n \;\; \textrm{odd},
\end{cases}
\end{align}
where $\P_k^{\ab}$ are extended to even functions on $(-\pi,\pi)$ by \eqref{P}.
Each $\Phi_{n}^{\ab}$ is an eigenfunction of $\mathbb{J}_{\ab}$, we have
$$
\mathbb{J}_{\alpha,\beta}\Phi_n^{\ab} = \lambda^{\ab}_{\langle n \rangle} \Phi_n^{\ab}, \qquad n \ge 0,
$$
with the notation
$$
\langle n \rangle = \bigg\lfloor \frac{n+1}2 \bigg\rfloor = \max\bigg\{ k \in \Z : k \le \frac{n+1}2\bigg\}.
$$
Thus $\mathbb{J}_{\ab}$, considered initially on $C_c^{2}((-\pi,\pi)\setminus \{0\})$,
has a natural self-adjoint extension in $L^2(d\m_{\ab})$, still denoted by $\mathbb{J}_{\ab}$ and given by
\begin{equation} \label{sresbb}
\mathbb{J}_{\ab} f = \sum_{n=0}^{\infty} \lambda_{\langle n \rangle}^{\ab} 
	\langle f, \Phi_n^{\ab} \rangle_{d\m_{\ab}} \Phi_n^{\ab}
\end{equation}
on the domain $\domain\mathbb{J}_{\ab}$ consisting of all functions $f\in L^2(d\m_{\ab})$ 
for which the defining series converges in $L^2(d\m_{\ab})$. 
Clearly, the spectral decomposition of $\mathbb{J}_{\ab}$ is given by \eqref{sresbb}.

The semigroup of operators generated by the square root of $\mathbb{J}_{\ab}$ will be denoted by
$\{\mathbb{H}_t^{\ab}\}$. We have, for $f \in L^2(d\m_{\ab})$ and $t\ge 0$,
\begin{equation} \label{Hser}
\mathbb{H}_t^{\ab}f = \exp\Big(-t\sqrt{\mathbb{J}_{\ab}}\Big)f
	= \sum_{n=0}^{\infty} \exp\Big(-t\sqrt{\lambda_{\langle n \rangle }^{\ab}}\,\Big)
	\langle f, \Phi_n^{\ab} \rangle_{d\m_{\ab}} \Phi_n^{\ab},
\end{equation}
the convergence being in $L^2(d\m_{\ab})$. 
In fact, for $t>0$ the last series converges pointwise for any $f \in L^p(wd\m_{\ab})$, $w \in A_p^{\ab}$,
$1\le p < \infty$, and defines a smooth function of $(t,\theta) \in (0,\infty)\times (-\pi,\pi)$.
Thus \eqref{Hser} can be regarded as the definition of $\{\mathbb{H}_t^{\ab}\}_{t>0}$ on the weighted spaces 
$L^p(wd\m_{\ab})$, $w\in A_p^{\ab}$, $1\le p < \infty$.
The integral representation of $\{\mathbb{H}_t^{\ab}\}_{t>0}$,
valid on the weighted $L^p$ spaces mentioned above,
is
$$
\mathbb{H}_t^{\ab}f(\theta) = \int_{-\pi}^{\pi} \mathbb{H}_t^{\ab}(\theta,\varphi)f(\varphi)\, 
d\m_{\ab}(\varphi),	\qquad \theta \in (-\pi,\pi), \quad t>0,
$$
with the symmetrized Jacobi-Poisson kernel
$$
\mathbb{H}_t^{\ab}(\theta,\varphi) =\sum_{n=0}^{\infty} 
	\exp\Big(-t\sqrt{\lambda_{\langle n \rangle }^{\ab}}\,\Big) \Phi_n^{\ab}(\theta)\Phi_n^{\ab}(\varphi).
$$
The last series converges absolutely and defines a smooth function of
$(t,\theta,\varphi)\in (0,\infty)\times(-\pi,\pi)^2$.

The central objects of our study are the following linear or sublinear operators associated with 
$\mathbb{J}_{\ab}$.
\begin{itemize}
\item[(i)] Symmetrized Riesz-Jacobi transforms of orders $N=1,2,\ldots$
$$
\mathbb{R}_N^{\ab}f = \sum_{n=0}^{\infty} \big(\lambda_{\langle n \rangle}^{\ab}\big)^{-N/2}
	\langle f, \Phi_n^{\ab} \rangle_{d\m_{\ab}} \mathbb{D}_{\ab}^N \Phi_n^{\ab}. 
$$
\item[(ii)] Multipliers of Laplace and Laplace-Stieltjes transform type
$$\mathbb{M}_m^{\ab}f(\theta)=\sum_{n=0}^{\infty} m\Big(\sqrt{\lambda_{\langle n \rangle}^{\ab}}\,\Big)
\langle f, \Phi_n^{\ab} \rangle_{d\m_{\ab}}\Phi_n^{\ab},$$
where either $m(z) = m_{\phi}(z) = \int_0^{\infty} z e^{-tz} \phi(t)\, dt$ with $\phi \in L^{\infty}(\R_+,dt)$
or $m(z) = m_{\nu}(z) = \int_{\R_+} e^{-tz} \, d\nu(t)$ with $\nu$ being a signed or complex Borel
measure on $\R_+=(0,\infty)$ whose total variation satisfies 
\begin{equation} \label{L-Sc}
\int_{\R_+} \exp\Big(-t \sqrt{\lambda_0^{\ab}}\,\Big) \, d|\nu|(t) < \infty.
\end{equation}
\item[(iii)] The symmetrized Jacobi-Poisson semigroup maximal operator
$$
\mathbb{H}^{\ab}_* f(\theta) = \big\|\mathbb{H}_t^{\ab}f(\theta)\big\|_{L^{\infty}(\R_+,dt)},
	\qquad \theta \in (-\pi,\pi).
$$
\item[(iv)] Symmetrized mixed square functions of arbitrary orders $M,N$
$$
\mathbb{G}_{M,N}^{\ab}(f)(\theta) = \big\|\partial_t^M \mathbb{D}_{\ab}^N 
\mathbb{H}_t^{\ab}f(\theta)\big\|_{L^2(\R_+,t^{2M+2N-1}dt)},\qquad \theta\in(-\pi,\pi),
$$
where $M,N=0,1,2,\ldots$ and $M+N>0$.
\end{itemize}
The operators $\mathbb{R}^{\ab}_N$ and $\mathbb{M}_m^{\ab}$ are well defined and bounded on $L^2(d\m_{\ab})$. 
In case of $\mathbb{M}_m^{\ab}$ this follows from Plancherel's theorem and the boundedness of $m$. 
The case of the Riesz transforms is covered by \cite[Proposition 4.4]{Symmetrized}.
We note that if $\alpha+\beta=-1$ then $0=\lambda_0^{\ab}$ is the eigenvalue of $\mathbb{J}_{\ab}$
and we actually need to interpret the series defining $\mathbb{R}_N^{\ab}$
as the sum over $n \ge 1$, in view of the identity
$\mathbb{D}_{\ab}\Phi_0^{\ab}\equiv 0$.
As for the remaining operators $\mathbb{H}_{*}^{\ab}$ and $\mathbb{G}_{M,N}^{\ab},$ their definitions are 
understood pointwise and make sense for general $f\in L^p(w d\mu_{\ab}), w \in A_p^{\ab}, 1\le p< \infty$, 
since, for such $f$, $\mathbb{H}_t^{\ab}f(\theta)$ is a smooth function of
$(t,\theta)\in (0,\infty)\times(-\pi,\pi)$.

\subsection{Symmetrized Jacobi function setting}
This context emerges from applying the symmetrization procedure from \cite{Symmetrized} to the system of
Jacobi functions. The extended measure space is $(-\pi,\pi)$ equipped with Lebesgue measure $d\t$.
We arrive at the symmetrized Laplacian
$$
\mathbb{L}_{\ab}=-\overline{\mathbb{D}}_{\ab}^2+\lambda_0^{\ab},
$$
where the associated derivative is given by
$$
\overline{\mathbb{D}}_{\ab}f=\frac{df}{d\theta}-\bigg(\frac{2\alpha+1}{4}\cot\frac{\theta}{2}-
\frac{2\beta+1}{4}\tan\frac{\theta}{2}\bigg)\check{f}=D_{\ab}f_{\textrm{even}}-D_{\ab}^*f_{\textrm{odd}},
$$
with $D_{\ab}$ and $D_{\ab}^*$ given on $(-\pi,\pi)$ by \eqref{der_ini}.

The orthonormal basis in $L^2(d\t)$ of eigenfunctions of $\mathbb{L}_{\ab}$ is $\{\Theta_n^{\ab}:n \ge 0\}$,
\begin{equation*}
\Theta_n^{\ab}(\t)=\frac{1}{\sqrt{2}}\begin{cases}\phi_{n/2}^{\ab}(\t),&\qquad  \textrm{$n$ even},\\
\sign(\theta)\phi_{(n-1)/2}^{\a+1,\b+1}(\t), &
\qquad \textrm{$n$ odd},
\end{cases}
\end{equation*}
where $\phi_n^{\ab}$ are even functions on $(-\pi,\pi)$ given by \eqref{Fi} and implicitly by \eqref{P}.
The corresponding eigenvalues are $\lambda_{\langle n \rangle}^{\ab}$,
$$
\mathbb{L}_{\alpha,\beta}\Theta_n^{\ab} = \lambda_{\langle n \rangle}^{\ab} \Theta_n^{\ab}, \qquad n \ge 0.
$$
Consequently, $\mathbb{L}_{\ab}$ has a natural self-adjoint extension from $C_c^{2}((-\pi,\pi)\setminus\{0\})$
to $L^2(-\pi, \pi)$, still denoted by $\mathbb{L}_{\ab}$, whose spectral decomposition is given by the
$\Theta_n^{\ab}$, see \eqref{sresbb}.

The semigroup of operators generated by the square root of $\mathbb{L}_{\ab}$ is denoted by
$\{\overline{\mathbb{H}}_{t}^{\ab}\}$. For $f \in L^2(-\pi,\pi)$ and $t \ge 0$ one has
\begin{equation} \label{Hfser}
\overline{\mathbb{H}}_{t}^{\ab}f = \exp\Big(-t\sqrt{\mathbb{L}^{\ab}}\Big)f = \sum_{n=0}^{\infty}
	\exp\Big(-t\sqrt{\lambda_{\langle n \rangle }^{\ab}}\,\Big)
	\langle f, \Theta_n^{\ab} \rangle \Theta_n^{\ab},
\end{equation}
with the convergence in $L^2(-\pi,\pi)$. 
Moreover, for $t>0$ the last series converges pointwise on $(-\pi,\pi)\setminus \{0\}$ and defines
a smooth function of $(t,\theta) \in (0,\infty)\times [(-\pi,\pi)\setminus\{0\}]$ provided that
$f \in L^p(w)$, $1 \le p < \infty$, and $w=w_{r,s}$ is an even double power weight satisfying
$r < p-1 + (\a+1/2)p$ and $s < p-1 + (\b+1/2)p$, with the last two inequalities weakened in case $p=1$; 
see \cite[Section 4]{L4} and \cite[Section 2]{L2}, and also \cite[Section 2]{Stempak}, for the
relevant arguments. This means, in particular, that \eqref{Hfser} defines $\overline{\mathbb{H}}_t^{\ab}f$,
$t>0$, on $L^p(w)$, $w \in B_p^{\ab}$, $1 \le p < \infty$.

As in the previous symmetrized setting,
we consider the following operators related to $\mathbb{L}_{\ab}$.
\begin{itemize}
\item[(i)] Symmetrized Riesz-Jacobi transforms of orders $N=1,2,\ldots$
$$
\overline{\mathbb{R}}_{N}^{\ab}f = \sum_{n=0}^{\infty} \big(\lambda_{\langle n \rangle}^{\ab}\big)^{-N/2}
	\langle f, \Theta_n^{\ab} \rangle \overline{\mathbb{D}}_{\ab}^N \Theta_n^{\ab}. 
$$
\item[(ii)] Multipliers of Laplace and Laplace-Stieltjes transform type
$$
\overline{\mathbb{M}}_m^{\ab}f(\theta)=\sum_{n=0}^{\infty} m\Big(\sqrt{\lambda_{\langle n
 \rangle}^{\ab}}\,\Big)	\langle f, \Theta_n^{\ab} \rangle\Theta_n^{\ab},
$$
where $m=m_{\phi}$ or $m=m_{\nu}$, with $m_{\phi}$ and $m_{\nu}$ as in Section \ref{poly}.
\item[(iii)] The symmetrized Jacobi-Poisson semigroup maximal operator
$$
\overline{\mathbb{H}}^{\ab}_{*} f(\theta) =
 \big\|\overline{\mathbb{H}}_{t}^{\ab}f(\theta)\big\|_{L^{\infty}(\R_+,dt)}, 
 	\qquad \theta \in (-\pi,\pi)\setminus \{0\}.
$$
\item[(iv)] Symmetrized mixed square functions of arbitrary orders $M,N$
$$
\overline{\mathbb{G}}_{M,N}^{\ab}(f)(\theta) = \big\|\partial_t^M \overline{\mathbb{D}}_{\ab}^N
 \overline{\mathbb{H}}_{t}^{\ab}f(\theta)\big\|_{L^2(\R_+,t^{2M+2N-1}dt)},
 	\qquad \theta\in(-\pi,\pi)\setminus \{0\},
$$
where $M,N=0,1,2,\ldots$ and $M+N>0$.
\end{itemize}
The operators $\overline{\mathbb{R}}^{\ab}_{N}$ and $\overline{\mathbb{M}}_m^{\ab}$ are well defined on
$L^2(-\pi,\pi)$, by \cite[Proposition 4.4]{Symmetrized} and Plancherel's theorem, respectively 
(in case $\a+\b=-1$ the bottom eigenvalue is $0$ and a proper interpretation of 
$\overline{\mathbb{R}}^{\ab}_N$ is needed, see the case of $\mathbb{R}^{\ab}_N$).
The definitions of $\overline{\mathbb{H}}_{*}^{\ab}$ and $\overline{\mathbb{G}}_{M,N}^{\ab}$ 
are understood pointwise for any $f\in L^p(w)$, $w\in B_p^{\ab}$, $1 \le p < \infty$.
This indeed makes sense since, for such $f$, $\overline{\mathbb{H}}_t^{\ab}f(\theta)$ is a smooth function
of $(t,\theta) \in (0,\infty) \times [(-\pi,\pi)\setminus \{0\}]$.

The following final observations are in order. The two symmetrized Jacobi settings are conjugated
by means of $\Psi^{\ab}$. We have
$$
\mathbb{L}_{\ab}(\Psi^{\ab}f) = \Psi^{\ab}\mathbb{J}_{\ab}f
$$
for suitable $f$ and analogous relations hold for the operators (i)-(iv). This allows us to transmit
certain
mapping properties of the relevant operators between the two frameworks, see Section \ref{ssec:32} below.
Moreover, the non-symmetrized settings, viz.\ those related to $(0,\pi)$, are naturally embedded in the
corresponding symmetrized ones. Consequently, essentially any results in the spirit of this paper
in the symmetrized situations can be projected suitably (by restricting them to even functions)
onto the non-symmetrized frameworks. Some new results following from this transference are presented
in Section \ref{ssec:33}.

\section{Main results} \label{sec:thm}

In this section we state the main results of the paper. For clarity and the reader's convenience,
we arrange them into three subsections corresponding to the two symmetrized Jacobi contexts, and the
non-symmetrized situations.

\subsection{Results in the symmetrized Jacobi polynomial setting} 
The theorem below is the principal result of the paper, since most of our results in the other settings
can be viewed as its consequences. 
\begin{thm} \label{thm:main}
Let $\alpha, \beta>-1$ and $w$ be an even weight on $(-\pi,\pi)$. Then the maximal operator
$\mathbb{H}_{*}^{\ab}$ and the square functions $\mathbb{G}_{M,N}^{\ab}$, $M,N=0,1,2,\ldots$, $M+N>0$,
are bounded on $L^p(w d\m_{\ab})$, $w\in A_p^{\ab}$, $1<p<\infty$ and from $L^1(w d\m_{\ab})$ to weak
$L^1(w d\m_{\ab})$, $w \in A_1^{\ab}$.
Furthermore, the Riesz transforms $\mathbb{R}_N^{\ab}$, $N=1,2,\ldots$ and the multipliers  
$\mathbb{M}_m^{\ab}$ extend uniquely to bounded linear operators on $L^p(w d\m_{\ab})$, $w\in A_p^{\ab}$, 
$1<p<\infty$ and from $L^1(w d\m_{\ab})$ to weak $L^1(w d\m_{\ab})$, $w \in A_1^{\ab}$.
\end{thm}

Following \cite{L1}, we now outline a reduction of the proof of Theorem \ref{thm:main} that allows us
to approach the problem by means of the powerful Calder\'on-Zygmund theory. Then the main difficulty
will be showing suitable kernel estimates, a tricky technical task to which we devote Section \ref{sec:ker}.
In the first step, proving Theorem \ref{thm:main} is reduced to showing analogous mapping properties
for suitably defined `restricted' operators related to the smaller space $((0,\pi),\mu_{\ab}^{+},|\cdot|)$.
This proceeds as follows.

Using \eqref{Phi} we decompose
\begin{align*}
\mathbb{H}_t^{\ab}(\theta,\varphi)&=\frac{1}{2}\sum_{n=0}^{\infty} 
\exp\Big(-t\sqrt{\lambda_{n}^{\ab}}\Big)\P_{n}^{\alpha,\beta}(\theta)\P_{n}^{\alpha,\beta}(\varphi)\\
&\quad+\frac{1}{8}\sin\theta\,\sin\varphi\sum_{n=0}^{\infty} 
\exp\Big(-t\sqrt{\lambda_{n+1}^{\ab}}\Big)\P_{n}^{\alpha+1,\beta+1}(\theta)\P_{n}^{\alpha+1,\beta+1}(\varphi)
\nonumber\\
&\equiv H_{t}^{\ab}(\theta,\varphi)+\widetilde{H}_{t}^{\ab}(\theta,\varphi). 
\end{align*}
The restriction of $H_{t}^{\ab}(\theta,\varphi)$ to $\theta,\varphi\in (0,\pi)$ coincides,
up to the factor $1/2$, with the standard (non-symmetrized) Jacobi-Poisson kernel 
related to $\J_{\ab}$ and studied recently in \cite{NS1,NS2,parameters}.
Furthermore, since each $\P_{n}^{\alpha,\beta}$ is an even function on $(-\pi,\pi),$ we see that 
$H_{t}^{\ab}(\theta,\varphi)$ and $\widetilde{H}_{t}^{\ab}(\theta,\varphi)$ are even and odd, respectively, 
functions both of $\theta$ and $\varphi$. Notice also that, since 
$\lambda_{n+1}^{\ab}=\lambda_{n}^{\alpha+1,\beta+1}$, we have 
\begin{equation*} 
\widetilde{H}_{t}^{\ab}(\theta,\varphi)=
\frac{1}{4}\sin\theta\,\sin\varphi\,H_{t}^{\alpha+1,\beta+1} (\theta,\varphi).
\end{equation*}
Using the sharp description of the Jacobi-Poisson kernel obtained in \cite[Theorem A.1]{NS2} and
\cite[Theorem 6.1]{parameters} it is straightforward to see that 
$\widetilde{H}_{t}^{\ab}(\theta,\varphi)$ is controlled pointwise by $H_{t}^{\ab}(\theta,\varphi)$, 
\begin{equation}\label{dom}
\big|\widetilde{H}_{t}^{\ab}(\theta,\varphi)\big|\lesssim H_{t}^{\ab}(\theta,\varphi), \qquad
	\theta, \varphi \in (-\pi,\pi), \quad t > 0.
\end{equation}

Next, we consider the operators acting on $L^{2}(d\mu_{\ab}^{+})$ and defined by
\begin{align*}
(H_t^{\ab})^{+}f &= \sum_{n=0}^{\infty} \exp\Big(-t\sqrt{\lambda_{n}^{\ab}}\Big)\langle f, \Phi_{2n}^{\ab} 
\rangle_{d\m_{\ab}^{+}} \Phi_{2n}^{\ab}, \\ 
(\widetilde{H}_t^{\ab})^{+}f &= \sum_{n=0}^{\infty} \exp\Big(-t\sqrt{\lambda_{n+1}^{\ab}}\Big)\langle f, 
\Phi_{2n+1}^{\ab} \rangle_{d\m_{\ab}^{+}} \Phi_{2n+1}^{\ab},
\end{align*}
for $t>0.$ Similarly as in the case of $\mathbb{H}_t^{\ab}$, the series defining $(H_t^{\ab})^+$ and
$(\widetilde{H}_t^{\ab})^{+}$ converge pointwise for any $f \in L^p(wd\mu_{\ab}^+)$, $w \in (A_p^{\ab})^+$,
$1 \le p < \infty$, and give rise to smooth functions of $(t,\theta) \in (0,\infty)\times (0,\pi)$.
The integral representations of $(H_t^{\ab})^{+}$ and $(\widetilde{H}_t^{\ab})^{+}$ are
\begin{align*}
(H_t^{\ab})^{+}f(\theta) &=\int_0^{\pi} H_t^{\ab}(\theta,\varphi)f(\varphi)\, d\m_{\ab}^{+}(\varphi),
	\\ 
(\widetilde{H}_t^{\ab})^{+}f(\theta) &=\int_0^{\pi} \widetilde{H}_t^{\ab}(\theta,\varphi)f(\varphi)
	\, d\m_{\ab}^{+}(\varphi).
\end{align*}

Denote 
$$
\de=\underbrace{\ldots \delta \delta_{\ab}^* \delta \delta_{\ab}^* \delta}_{N\; \textrm{components}},\qquad 
\ko=\underbrace{\ldots \delta_{\ab}^* \delta \delta_{\ab}^* \delta \delta_{\ab}^*}_{N\; \textrm{components}},
$$ 
with the natural convention for the case $N=0$.
These derivatives correspond to the action of $\mathbb{D}_{\ab}^N$ on even and odd functions, respectively.
In particular, 
$$
\mathbb{D}_{\ab}^N f=\de f_{\textrm{even}}+\ko f_{\textrm{odd}}.
$$
Now we are ready to define the `restricted' operators we need:
\begin{itemize}
\item[(i)]
\begin{align*}
(R_N^{\ab})^{+}f&=\sum_{n=1}^{\infty} (\lambda_{n}^{\ab})^{-N/2}	\langle f, \Phi_{2n}^{\ab} 
\rangle_{d\m_{\ab}^{+}} \de \Phi_{2n}^{\ab},\qquad f\in L^{2}(d\mu_{\ab}^{+}),\\
(\widetilde{R}_N^{\ab})^{+}f&=\sum_{n=0}^{\infty} (\lambda_{n+1}^{\ab})^{-N/2}	\langle f, \Phi_{2n+1}^{\ab} 
\rangle_{d\m_{\ab}^{+}} \ko \Phi_{2n+1}^{\ab},\qquad f\in L^{2}(d\mu_{\ab}^{+}),
\end{align*}

\item[(ii)]
\begin{align*}
(M_m^{\ab})^{+}f&=\sum_{n=0}^{\infty} m\Big(\sqrt{\lambda_{n}^{\ab}}\Big)	\langle f, \Phi_{2n}^{\ab} 
\rangle_{d\m_{\ab}^{+}}\Phi_{2n}^{\ab},\qquad f\in L^{2}(d\mu_{\ab}^{+}),\\
(\widetilde{M}_m^{\ab})^{+}f&=\sum_{n=0}^{\infty} {m}\Big(\sqrt{\lambda_{n+1}^{\ab}}\Big)	\langle f, 
\Phi_{2n+1}^{\ab} \rangle_{d\m_{\ab}^{+}}\Phi_{2n+1}^{\ab},\qquad f\in L^{2}(d\mu_{\ab}^{+}),
\end{align*}

\item[(iii)]
\begin{align*}
({H}^{\ab}_*)^{+} f(\theta) &=\big\|(H_t^{\ab})^{+}f(\theta)\big\|_{L^{\infty}(\R_+,dt)},\\
(\widetilde{H}^{\ab}_{*})^{+} f(\theta) &=\big\|(\widetilde{H}_t^{\ab})^{+}f(\theta)
\big\|_{L^{\infty}(\R_+,dt)},
\end{align*}

\item[(iv)]
\begin{align*}
({G_{M,N}^{\ab}})^{+}(f)(\theta) &= \big\|\partial_t^M \de 
(H_t^{\ab})^{+}f(\theta)\big\|_{L^2(\R_+,t^{2M+2N-1}dt)},\\
(\widetilde{G}_{M,N}^{\ab})^{+}(f)(\theta) &= \big\|\partial_t^M \ko 
(\widetilde{H}_t^{\ab})^{+}f(\theta)\big\|_{L^2(\R_+,t^{2M+2N-1}dt)}.
\end{align*}
\end{itemize}
The proof of Theorem \ref{thm:main} reduces to showing the following statement, see \cite[Section 2]{L1}
for the details.
\begin{thm} \label{thm:main'}
Let $\alpha, \beta>-1$. Then the operators $({H}_{*}^{\ab})^{+}$, $(\widetilde{H}_{*}^{\ab})^{+}$, 
$(G_{M,N}^{\ab})^{+}$, $(\widetilde{G}_{M,N}^{\ab})^{+}$, $M,N=0,1,2,\ldots$, $M+N>0$, are bounded on $L^p(w 
d\m_{\ab}^{+})$, $w \in (A_p^{\ab})^{+}$, $1<p<\infty$, and from $L^1(w d\m_{\ab}^{+})$ 
to weak $L^1(w d\m_{\ab}^{+})$, $w \in (A_1^{\ab})^{+}$.
Furthermore, the operators $({R}_N^{\ab})^{+}$, $(\widetilde{R}_N^{\ab})^{+}$, $N=1,2,\ldots$, $(M_m^{\ab})^{+}$ 
and $(\widetilde{M}_m^{\ab})^{+}$ extend uniquely to bounded linear operators on $L^p(w d\m_{\ab}^{+})$, $w\in 
(A_p^{\ab})^{+}$, $1<p<\infty$, and from $L^1(w d\m_{\ab}^{+})$ 
to weak $L^1(w d\m_{\ab}^{+})$, $w\in (A_1^{\ab})^{+}$.
\end{thm}

A part of Theorem \ref{thm:main'} is covered by the existing literature. More precisely, $(H_*^{\ab})^+$
and $(M_m^{\ab})^+$ were proved to possess the desired boundedness properties in 
\cite[Corollary 5.2]{parameters}, and $(R_N^{\ab})^+$ was treated in \cite[Proposition 3.7]{CNS}
(here we implicitly identify these `restricted' operators with the corresponding operators
in the non-symmetrized Jacobi polynomial setting). Moreover, in view of \eqref{dom}, the mapping properties
of $(H_*^{\ab})^+$ in question imply the same mapping properties for $(\widetilde{H}_*^{\ab})^+$.
Therefore, to prove Theorem \ref{thm:main'} it remains to deal with $(\widetilde{R}_N^{\ab})^{+}$,
$(\widetilde{M}_m^{\ab})^{+}$, $(G_{M,N}^{\ab})^{+}$ and $(\widetilde{G}_{M,N}^{\ab})^{+}$.
Finally, we remark that Theorem \ref{thm:main'} was stated and justified in \cite{L1} under the restriction
$\ab \ge -1/2$. Here, apart from extending that result, we take the opportunity to simplify and shorten
the reasoning given in \cite{L1} by means of the techniques elaborated recently in \cite{parameters}.

The part of Theorem \ref{thm:main'} that needs to be proved is a consequence of a more general result
stated below. In case of $(\widetilde{R}_N^{\ab})^+$ and $(\widetilde{M}_m^{\ab})^+$ this claim 
follows directly
from the general Calder\'on-Zygmund theory for spaces of homogeneous type. The implication in case of
$(G_{M,N}^{\ab})^{+}$ and $(\widetilde{G}_{M,N}^{\ab})^+$ is easily justified by arguments analogous
to those given in the proof of \cite[Corollary 2.5]{NS1}. Notice that
$(G_{M,N}^{\ab})^{+}$ and $(\widetilde{G}_{M,N}^{\ab})^+$ are not linear, but can be naturally interpreted
as vector-valued linear operators taking values in the Banach space $L^2(\R_+,t^{2M+2N-1}dt)$.
\begin{thm} \label{thm:main2}
Assume that $\alpha, \beta>-1$. The operators $(\widetilde{R}_N^{\ab})^{+}$, $N=1,2,\ldots$ and 
$(\widetilde{M}_m^{\ab})^{+}$ are Calder\'on-Zygmund operators in the sense of the space of homogeneous
type $((0,\pi),d\m_{\ab}^{+},{|\cdot|})$. 
Further, the operators $(G_{M,N}^{\ab})^{+}$ and $(\widetilde{G}_{M,N}^{\ab})^{+}$, $M,N=0,1,2,\ldots$, 
$M+N>0$, viewed as vector-valued linear operators, are Calder\'on-Zygmund operators, 
in the sense of the same space of homogeneous type,
associated with the Banach space $L^2(\R_+,t^{2M+2N-1}dt)$.
\end{thm}
 
The proof of Theorem \ref{thm:main2} splits naturally into showing the following three results. 
The first two of them are essentially contained in \cite{L1}.
\begin{prop} \label{prop:L2}
Let $\alpha,\beta > -1$. The operators $(\widetilde{R}_N^{\ab})^{+}$, $N=1,2,\ldots$,  
$(\widetilde{M}_m^{\ab})^{+}$, $(G_{M,N}^{\ab})^{+}$ and $(\widetilde{G}_{M,N}^{\ab})^{+}$,
$M,N=0,1,2,\ldots$, $M+N>0$, are bounded on $L^2(d\m_{\ab}^{+})$.
In particular, the operators $(G_{M,N}^{\ab})^{+}$ and $(\widetilde{G}_{M,N}^{\ab})^{+}$,
$M,N=0,1,2,\ldots$, $M+N>0$, viewed as vector-valued linear operators,
are bounded from $L^2(d\m_{\ab}^{+})$ to the Bochner-Lebesgue space
$L^2_{L^2(\R_+,t^{2M+2N-1}dt)}(d\m_{\ab}^{+})$.
\end{prop}

\begin{proof}
It suffices to observe that all the arguments given in the proof of \cite[Proposition 2.4]{L1} remain valid 
for the full range $\ab>-1$.
\end{proof}

For $\theta, \varphi\in(0,\pi)$ define the kernels
\begin{align*}
\widetilde{M}^{\ab}_{\phi}(\theta,\varphi) & = -\int_0^{\infty}
	\partial_{t}\widetilde{H}_t^{\ab}(\theta,\varphi) \phi(t)\,dt,  \\
\widetilde{M}^{\ab}_{\nu}(\theta,\varphi) & = \int_{(0,\infty)}
	\widetilde{H}_t^{\ab}(\theta,\varphi)\,d\nu(t),  \\	
\widetilde{R}_N^{\ab}(\theta,\varphi) & = \frac{1}{\Gamma(N)}\int_0^{\infty} 
	\ko \widetilde{H}_t^{\ab}(\theta,\varphi) t^{N-1}\,dt, \qquad N\ge 1. 
\end{align*}
Here and elsewhere the derivatives $\de$ and $\ko$ act always on $\theta$ variable. The next result 
establishes the weak association (see \cite[p.\,262]{L1}) of the operators in question with the corresponding integral kernels.
\begin{prop} \label{prop:assoc}
Let $\alpha,\beta> -1$. The operators $(\widetilde{R}_N^{\ab})^{+}$, $N=1,2,\ldots$ and 
$(\widetilde{M}_m^{\ab})^{+}$ are associated in the Calder\'on-Zygmund theory sense 
with the following scalar-valued kernels:
$$
(\widetilde{R}_N^{\ab})^{+} \sim \widetilde{R}_N^{\ab}(\theta,\varphi), \qquad
(\widetilde{M}_m^{\ab})^{+} \sim \widetilde{M}_m^{\ab}(\theta,\varphi), 
$$
where $\widetilde{M}_m^{\ab}(\theta,\varphi)$ is equal either to 
$\widetilde{M}_{\phi}^{\ab}(\theta,\varphi)$ or $\widetilde{M}_{\nu}^{\ab}(\theta,\varphi)$, 
depending on whether $m=m_{\phi}$ or $m=m_{\nu}$, respectively.
Further, the operators $(G_{M,N}^{\ab})^{+}$, $(\widetilde{G}_{M,N}^{\ab})^{+},$
$M,N=0,1,2,\ldots$, $M+N>0$, viewed as vector-valued linear operators,
are associated with the following kernels taking values in $L^2(\R_+,t^{2M+2N-1}dt)$:
$$
(G_{M,N}^{\ab})^{+}  \sim \{\partial_t^{M}\delta_N^{\emph{\textrm{even}}} H_t^{\ab}(\theta,\varphi)\}_{t>0}, 
\qquad
(\widetilde{G}_{M,N}^{\ab})^{+}  \sim \{\partial_t^{M}\delta_N^{\emph{\textrm{odd}}} 
\widetilde{H}_t^{\ab}(\theta,\varphi)\}_{t>0}.
$$
\end{prop}

\begin{proof}
When $\ab \ge -1/2$, this is contained in \cite[Proposition 2.5]{L1}. It is enough to notice that
the proof given in \cite{L1} works in fact for all $\ab > -1$, provided that we combine it with
Proposition \ref{prop:L2} and the estimates obtained independently in Section 5 below.
\end{proof}

In order to complete the proof of Theorem \ref{thm:main2} we must show that the kernels in question
satisfy the so-called standard estimates, see \eqref{gr}-\eqref{grad} in Section \ref{sec:ker}.
This is contained in the next statement.
\begin{thm} \label{thm:stand}
Assume that $\alpha,\beta>-1$. The scalar-valued kernels from Proposition \ref{prop:assoc}
satisfy the standard estimates \eqref{gr} and \eqref{grad} with $\mathbb{B}=\mathbb{C}$.
The vector-valued kernels from Proposition \ref{prop:assoc} satisfy
the standard estimates \eqref{gr} and \eqref{grad} with $\mathbb{B}=L^2(\R_+,t^{2M+2N-1}dt)$.
\end{thm}
The proof of Theorem \ref{thm:stand} requires an involved analysis and is given in Section \ref{sec:ker}.

\subsection{Results in the symmetrized Jacobi function setting} \label{ssec:32}
We now state a counterpart of Theorem \ref{thm:main} in the symmetrized Jacobi function context.
Actually, it is a consequence of Theorem \ref{thm:main} obtained by means of the already announced
transference, see the end of Section~\ref{sec:prel}. For the sake of clarity, we restrict here to
even double power weights, since this class is invariant under multiplication by powers of $\Psi^{\ab}$.
Nevertheless, Theorem \ref{thm:main} combined with the transference method allows one to conclude
results with more general weights as well. This is left to interested readers.
\begin{thm} \label{thm:mainf}
Let $\ab>-1$ and $1<p<\infty$. Then the maximal operator $\overline{\mathbb{H}}_{*}^{\ab}$ and the square 
functions $\overline{\mathbb{G}}_{M,N}^{\ab}$, $M,N=0,1,2,\ldots$, $M+N>0$, are bounded on $L^p(w)$,
$w\in B_p^{\ab}$. 
Furthermore, the Riesz transforms $\overline{\mathbb{R}}_{N}^{\ab}$, $N=1,2,\ldots$ and the multipliers  
$\overline{\mathbb{M}}_m^{\ab}$ extend uniquely to bounded linear operators on $L^p(w)$, $w\in B_p^{\ab}$.
\end{thm}

\begin{proof}
Consider first $\overline{\mathbb{H}}_{*}^{\ab}$.
Observe that $\overline{\mathbb{H}}_{*}^{\ab}f = \Psi^{\ab}\mathbb{H}_{*}^{\ab}(\Psi^{-\alpha-1,-\beta-1}f)$.
This relation, combined in a straightforward manner with Theorem \ref{thm:main} specified to $p>1$ and
even double power weights in $A_p^{\ab}$, implies weighted $L^p$-boundedness for 
$\overline{\mathbb{H}}_{*}^{\ab}$ with weights belonging to $B_p^{\ab}$.

The other operators are treated similarly. The only difference is that in the cases of 
$\overline{\mathbb{R}}_{N}^{\ab}$ and  $\overline{\mathbb{M}}_m^{\ab}$ one shows the desired weighted 
boundedness on the linear span of the $\Theta_n^{\ab}$, $n \ge 0$,
and then uses a density argument to extend it to the whole $L^p(w)$.
\end{proof}

We remark that the weak type boundedness results contained in Theorem \ref{thm:main} and corresponding to the
case $p=1$ cannot be transferred in a similar spirit to the present context. On the other hand,
it is perhaps of interest to specify Theorem \ref{thm:mainf} to the unweighted situation.
Notice that a restriction on $p$ comes into play when $\a$ or $\b$ is less than $-1/2$,
and this is due to the so-called pencil phenomenon occurring in the Jacobi function settings,
see \cite{L2,L3,L4}.
\begin{cor}
Let $\ab > -1$ and $1 < p < \infty$. Then the operators from Theorem \ref{thm:mainf} are bounded on
$L^p(-\pi,\pi)$ or extend to such operators provided that either $\ab \ge -1/2$ or $\min(\ab) < -1/2$ and
$p$ is restricted by the condition
$$
-\min(\ab)-\frac{1}2 < \frac{1}{p} < \min(\ab) + \frac{3}2.
$$
\end{cor}

Finally, we note that Theorem \ref{thm:mainf} generalizes \cite[Lemma 3.6]{L4}, where 
$\overline{\mathbb{R}}_N^{\ab}$, $N=1,2,\ldots$, $\a+\b \neq -1$,
were proved to be bounded on unweighted $L^p$ by completely
different methods.

\subsection{Results in the non-symmetrized Jacobi settings} \label{ssec:33}
In this subsection we gather new results in the non-symmetrized Jacobi situations, 
most of which can be seen as consequences of Theorems \ref{thm:main} and \ref{thm:mainf}.

Let $\{\mathcal{H}_t^{\ab}\}_{t>0}$ be the Jacobi-Poisson semigroup corresponding to the polynomial system
$\{\P_n^{\ab}: n \ge 0\}$. This semigroup and objects based upon it were investigated in
\cite{CNS,L3,Nowak&Roncal,NS1,parameters}, among others. In particular, various square functions involving 
$\{\mathcal{H}_t^{\ab}\}$ were studied in these papers. Here we consider another square function,
defined via the interlaced higher-order derivatives
$\delta_N^{\textrm{even}} = \ldots \delta \delta_{\ab}^* \delta \delta_{\ab}^* \delta$
($N$ components).
\begin{prop} \label{prop:genfg}
Let $\ab > -1$. Then, for each $M,N=0,1,2,\ldots$ such that $M+N>0$, the square function
$$
f \longmapsto \big\|\partial_t^M \delta_N^{\textrm{even}} \mathcal{H}_t^{\ab}f\big\|_{L^2(\R_+,t^{2M+2N-1}dt)}
$$
is bounded on $L^p(wd\mu_{\ab}^+)$, $w \in (A_p^{\ab})^+$, $1< p < \infty$, and from $L^1(wd\mu_{\ab}^+)$
to weak $L^1(wd\mu_{\ab}^+)$, $w \in (A_1^{\ab})^+$.
\end{prop}

\begin{proof}
It is enough to observe that the square function defined in Proposition \ref{prop:genfg} is,
up to the factor $1/2$, the `restricted' operator $(G_{M,N}^{\ab})^{+}$. Thus the conclusion follows
immediately from Theorem \ref{thm:main'}.
\end{proof}

Alternatively, to prove Proposition \ref{prop:genfg} one could argue via Theorem \ref{thm:main} by
considering the action of $\mathbb{G}_{M,N}^{\ab}$ on even functions.

We now pass to the less explored non-symmetrized Jacobi function setting.
Let $\{\overline{\mathcal{H}}_t^{\ab}\}_{t>0}$ be the Jacobi-Poisson semigroup in this context,
see Sections 2 in \cite{L2,L3}. We have 
\begin{equation*}
\overline{\mathcal{H}}_{t}^{\ab}f = \sum_{n=0}^{\infty} \exp\Big(-t\sqrt{\lambda_{n}^{\ab}}\Big)
	\langle f, \phi_{n}^{\ab} \rangle \phi_{n}^{\ab},
\end{equation*}
the series being convergent pointwise on $(0,\pi)$ and defining a smooth function of 
$(t,\theta) \in (0,\infty)\times (0,\pi)$ provided that $f \in L^p(w)$ for some $1 \le p < \infty$ and
$w=w_{r,s}$ is a double power weight on $(0,\pi)$ satisfying $r < p-1 + (\a+1/2)p$ and $s < p-1 + (\b+1/2)p$,
with the last two inequalities weakened in case $p=1$; see \cite[Proposition 3.1]{Stempak}.
Further, we set 
$D_N^{\textrm{even}}=\ldots D_{\ab} D_{\ab}^* D_{\ab} D_{\ab}^* D_{\ab}$ ($N$ components).

We consider the following operators related to the system $\{\phi_n^{\ab}: n \ge 0\}$ on $(0,\pi)$.
\begin{itemize}
\item[(a)] The Riesz-Jacobi transforms of orders $N=1,2,\ldots$
$$
f\longmapsto\sum_{n=1}^{\infty} (\lambda_{n}^{\ab})^{-N/2}	\langle f, \phi_{n}^{\ab} \rangle
  D_{\ab}^N \phi_{n}^{\ab},\qquad f\in L^{2}(0,\pi).
$$
\item[(b)] The interlaced Riesz-Jacobi transforms of orders $N=1,2,\ldots$
$$
f\longmapsto \sum_{n=1}^{\infty}  (\lambda_{n}^{\ab})^{-N/2}
	\langle f, \phi_n^{\ab} \rangle D_N^{\textrm{even}} \phi_n^{\ab}, \qquad f\in L^{2}(0,\pi).
$$
\item[(c)] Multipliers of Laplace and Laplace-Stieltjes  transform type
$$
f\longmapsto \sum_{n=0}^{\infty} m\Big(\sqrt{\lambda_{n}^{\ab}}\Big)	\langle f, \phi_{n}^{\ab} \rangle
	 \phi_{n}^{\ab},\qquad f\in L^{2}(0,\pi),
$$
where $m=m_{\phi}$ or $m=m_{\nu}$, with $m_{\phi}$ and $m_{\nu}$ as in Section \ref{poly}.
\item[(d)] The Jacobi-Poisson semigroup maximal operator
$$
f\longmapsto\big\|\overline{\mathcal{H}}_{t}^{\ab}f\big\|_{L^{\infty}(\R_+,dt)}.
$$
\item[(e)] Mixed square functions of arbitrary orders $M,N$
$$
f\longmapsto \big\|\partial_t^M  D_{\ab}^N \overline{\mathcal{H}}_{t}^{\ab}f\big\|_{L^2(\R_+,t^{2M+2N-1}dt)},
$$
where $M,N=0,1,2,\ldots$ and $M+N>0$.
\item[(f)] Interlaced mixed square functions of arbitrary orders $M,N$
$$
f \longmapsto \big\|\partial_t^M D_N^{\textrm{even}}
	 \overline{\mathcal{H}}_{t}^{\ab}f\big\|_{L^2(\R_+,t^{2M+2N-1}dt)},
$$
where $M,N=0,1,2,\ldots$ and $M+N>0$.
\end{itemize}
The operators in (a)-(c) are well defined in $L^2(0,\pi)$, with the series being convergent there.
The remaining operators are well defined pointwise for $f \in L^p(w)$, where $1 \le p < \infty$ and
$w$ is any weight from $(B_p^{\ab})^+$. The latter class consists of restrictions of weights $w_{r,s}$ in
$B_p^{\ab}$ to the interval $(0,\pi)$, so the definition in terms of ranges of $r$ and $s$ remains the same.
\begin{prop} \label{prop:genJf}
Let $\ab > -1$ and $1<p<\infty$. Then the operators (d)-(f) are bounded on $L^p(w)$, $w \in (B_p^{\ab})^+$.
Further, the operators (a)-(c) extend uniquely to bounded linear operators on $L^p(w)$,
$w \in (B_p^{\ab})^+$.
\end{prop}

\begin{proof}
Because of a transference argument analogous to that from the proof of Theorem \ref{thm:mainf},
see the comment at the end of Section \ref{ssec:21}, it is sufficient to check that the counterparts
of the operators (a)-(f) in the Jacobi polynomial setting are (or extend to) bounded operators on
$L^p(wd\mu_{\ab}^+)$, $w \in (A_p^{\ab})^+$, $1<p<\infty$. This, in turn, is done with the aid of
the already known results. Indeed, counterparts of (a) and (c)-(e) 
related to $\{\P_n^{\ab}:n \ge 0\}$ are
covered by \cite[Corollary 5.2]{parameters}, the counterpart of (b) is taken care of by
\cite[Proposition 3.7]{CNS}, and finally, in case of (f), Proposition \ref{prop:genfg} does the job.
\end{proof}

Alternatively (but less directly), to prove Proposition \ref{prop:genJf}
in cases of (b)-(d) and (f) one could argue by means of 
Theorem \ref{thm:mainf}, restricting the action of the operators appearing there to even functions.
We leave details to interested readers.

We point out that $L^p$-boundedness of some of the operators in Proposition \ref{prop:genJf} was
studied earlier. In particular, double power weighted $L^p$-boundedness of the first order Riesz-Jacobi
transform, (a) with $N=1$, was obtained in \cite[Theorem 4.3]{Stempak} by means of Muckenhoupt's
general multiplier-transplantation theorem.
Unweighted $L^p$-boundedness of the interlaced Riesz-Jacobi transforms
(b) and the maximal operator (d) was shown in \cite[Proposition 4.2]{L2} and \cite[Proposition 2.2]{L2},
respectively. We also mention that \cite{L2} contains $L^p$ results for variants of Riesz-Jacobi
transforms that are different from the operators in (a) and (b), and in \cite{L3} one can find $L^p$ results
for variants of square functions different from those in (e) and (f).
Weighted $L^p$ boundedness results for the variants just mentioned were obtained recently in \cite{ABCSS},
though under the restriction $\ab \ge -1/2$.

\section{Preparatory results}\label{sec:tech}

In this section we collect some technical results that will be needed in the proof of Theorem~\ref{thm:stand}.
Some of these results are taken from other papers, nevertheless we state them here for the sake of the
reader's convenience.

We point out that a crucial role in our developments is played by technical results from \cite{parameters}.
There the authors established an integral formula for the Jacobi-Poisson kernel that is valid for the full
range of parameters $\ab > -1$, extending an analogous result from \cite{NS1} burdened by the restriction
$\ab \ge -1/2$. The representation of the Jacobi-Poisson kernel enabled the authors of 
\cite{NS1, parameters} to elaborate concise and elegant techniques of estimating integral kernels of various
harmonic analysis operators. We will take advantage of these methods in Section \ref{sec:ker}.

In the sequel we shall use the notation from \cite{parameters}. Thus $d\Pi_{\a}$ and $d\Pi_{\a,K}$
are certain measures we do not need to define explicitly here (see \cite{parameters} for the definitions)
and $q$ is a function of $\theta,\varphi \in (0,\pi)$ and $u,v\in [-1,1]$ valued in $[0,2]$ (in particular,
non-negative and bounded) and given by
$$
q(\theta,\varphi,u,v) = 1 - u \sin\frac{\theta}2 \sin\frac{\varphi}2
	- v \cos\frac{\theta}2 \cos\frac{\varphi}2.
$$
For the sake of brevity we shall omit the arguments and write shortly
$$
\q:=q(\theta,\varphi,u,v).
$$

The following crucial estimates allow one to control mixed derivatives of the Jacobi-Poisson kernel by suitable double-integral expressions.

\begin{lem}[{\cite[Corollary 3.5]{parameters}}]\label{cor:t1}
Let $M,N \in \mathbb{N}$ and $L \in \{0 , 1\}$ be fixed. 
The following estimates hold uniformly in $t\in (0,1]$ and $\t,\vp \in (0,\pi)$.
\begin{itemize}

\item[(i)]
If $\a,\b \ge -1\slash 2$, then 
\begin{align*}
\big| 
	 \partial_\vp^L \partial_\t^N \partial_t^M H_{t}^{\ab}(\t,\vp)  
\big|  
	\lesssim \iint\frac{\pia \, \pib}{(t^2 + \q)^{   \a + \b + 3\slash 2 + (L+N+M)\slash 2  }}.
\end{align*}

\item[(ii)]
If $-1 < \a < -1\slash 2 \le \b$, then 
\begin{align*}
\big| 
	 \partial_\vp^L \partial_\t^N \partial_t^M H_{t}^{\ab}(\t,\vp)  \big|  
	&\lesssim 
1 + \sum_{K=0,1}\sum_{k=0,1,2} \bigg( \st+\svp \bigg)^{Kk}  \\ 
& \qquad \qquad \times
\iint\frac{\piK \, \pib}{(t^2 + \q)^{   \a + \b + 3\slash 2 + (L+N+M + Kk)\slash 2  }}.
\end{align*}

\item[(iii)]
If $-1 < \b < -1\slash 2 \le \a$, then 
\begin{align*}
\big| 
	\partial_\vp^L \partial_\t^N \partial_t^M H_{t}^{\ab}(\t,\vp) 
\big| 
	&\lesssim 
1 + \sum_{R=0,1} \sum_{r=0,1,2} \bigg( \ct + \cvp \bigg)^{Rr} \\
& \qquad \qquad \times
\iint \frac{\pia \, \piR}{(t^2 + \q)^{   \a + \b + 3\slash 2 + (L+N+M +Rr)\slash 2  }}.
\end{align*}

\item[(iv)] 
If $-1 < \a,\b < -1\slash 2$, then
\begin{align*}
\big| 
	\partial_\vp^L \partial_\t^N \partial_t^M H_{t}^{\ab}(\t,\vp) \big| 
	&\lesssim 
1 + \sum_{K,R=0,1}\sum_{k,r=0,1,2} \bigg( \st+\svp \bigg)^{Kk} \bigg( \ct + \cvp \bigg)^{Rr} 
\\ 
& \qquad \qquad \times
\iint\frac{\piK \, \piR}{(t^2 + \q)^{   \a + \b + 3\slash 2 + (L+N+M + Kk +Rr)\slash 2  }}.
\end{align*}
\end{itemize}
\end{lem}

Given $\theta \in (0,\pi)$ and $r > 0$, denote by $B(\theta,r)$ the standard ball on the real line
restricted to the interval $(0,\pi)$, i.e.\ $B(\theta,r)=(\theta-r,\theta+r)\cap (0,\pi)$.
The next result establishes a bridge between double integrals like those above and the standard estimates
we must prove. We remark that only the cases $p \in \{1,2,\infty\}$ will be needed for our purposes.

\begin{lem}[{\cite[Lemma 3.7]{parameters}}]\label{lem:finbridge}
Let $K,R \in \{ 0,1 \}$, $k,r \in \{ 0,1,2 \}$, $W \ge 1$, $s\ge 0$ and $1 \le p \le \infty$ be fixed.
Consider a function $\Upsilon^{\ab}_s(t,\t,\vp)$ defined on $(0,1) \times (0,\pi) \times (0,\pi)$ 
in the following way.
\begin{itemize}
\item[(i)]
For $\a,\b \ge -1\slash 2$,
\begin{align*}
\Upsilon^{\ab}_s(t,\t,\vp) := 
\iint
\frac{\pia \, \pib }{(t^2 + \q)^{ \a+\b+3\slash 2 + W\slash (2p) + s\slash 2  }}.
\end{align*}
\item[(ii)]
For $-1 < \a < -1\slash 2 \le \b$,
\begin{align*}
\Upsilon^{\ab}_s(t,\t,\vp) := 
\bigg( \st + \svp \bigg)^{Kk} \iint
\frac{\piK \, \pib }{(t^2 + \q)^{ \a+\b+3\slash 2 + W\slash (2p) +Kk\slash 2 + s\slash 2  }}.
\end{align*}
\item[(iii)]
For $-1 < \b < -1\slash 2 \le \a$,
\begin{align*}
\Upsilon^{\ab}_s(t,\t,\vp) := 
\bigg( \ct + \cvp \bigg)^{Rr} \iint
\frac{\pia \, \piR }{(t^2 + \q)^{ \a+\b+3\slash 2 + W\slash (2p) +Rr\slash 2 + s\slash 2  }}.
\end{align*}
\item[(iv)] 
For $-1 < \a,\b < -1\slash 2$,
\begin{align*}
\Upsilon^{\ab}_s(t,\t,\vp) := &
\bigg( \st+\svp \bigg)^{Kk} \bigg( \ct + \cvp \bigg)^{Rr} \\
& \quad \times \iint
\frac{\piK \, \piR }{(t^2 + \q)^{ \a+\b+3\slash 2 + W\slash (2p) +Kk\slash 2 + Rr\slash 2 + s\slash 2  }}.
\end{align*}
\end{itemize}
Then the estimate
\[
\big\| 1 + \Upsilon^{\ab}_s(t,\t,\vp) \big\|_{ L^p((0,1),t^{W-1}dt) } 
\lesssim
\frac{1}{|\t-\vp|^s} \; \frac{1}{\mu_{\ab}^+(B(\t,|\t-\vp|))}
\]
holds uniformly in $\t,\vp \in (0,\pi)$, $\t\ne\vp$.
\end{lem}

We shall need as well a long-time counterpart of Lemma \ref{lem:finbridge}.

\begin{lem}[{\cite[Corollary 3.9]{parameters}}] \label{cor:tL}
Let $\ab > -1$, $M,N \in \N$, $L \in \{0,1\}$, $W \ge 1$ and $1 \le p \le \infty$ be fixed.
Then
$$
\bigg\|\sup_{\t,\vp \in (0,\pi)} \big|\partial_{\vp}^{L} \partial_{\t}^N \partial_t^M H_t^{\ab}(\t,\vp) \big|
	\bigg\|_{L^p((1,\infty),t^{W-1}dt)} < \infty,
$$
excluding the cases when simultaneously $\a+\b+1=0$ and $M=N=L=0$ and $p<\infty$.
\end{lem}

The following strengthened special case of Lemma \ref{cor:tL} will be used to treat kernels
associated with the Laplace-Stieltjes type multipliers.
\begin{lem}[{\cite[Corollary 3.10]{parameters}}]\label{cor:tLStieltjes}
Let $\ab > -1$ and 
$L,N \in \{0,1\}$ be fixed.
Then
$$
\bigg\| e^{t \left| \frac{\a + \b + 1}{2} \right|} \sup_{\t,\vp \in (0,\pi)} 
	\big|\partial_{\vp}^{L} \partial_{\t}^N H_t^{\ab}(\t,\vp) \big|
	\bigg\|_{L^\infty((1,\infty),dt)} < \infty.
$$
\end{lem}

The above tools from \cite{parameters} are yet insufficient to deal with all
expressions arising in the analysis of Section \ref{sec:ker}.
The next two lemmas eventually complete the toolbox.
\begin{lem}\label{lem:gr}
Let $\ab> -1$, $M,N\in \N$, $L,\gamma_1,\gamma_2 \in \{0,1\}$,  $\gamma_1+\gamma_2\ge 1$, $W \ge 1$ 
and $1 \le p \le \infty$ be fixed and such that $W/p+\gamma_1+\gamma_2-L-N-M\ge 2$.
Then the estimate
$$
(\sin\theta)^{\gamma_1}(\sin\varphi)^{\gamma_2}\Big\|\partial_{\vp}^{L} \partial_{\t}^N \partial_t^M H_t^{\a+1, \b+1}(\t,\vp) \Big\|_{L^p(\mathbb{R}_+,t^{W-1} dt)}\lesssim
\frac{1}{\mu_{\ab}^+(B(\t,|\t-\vp|))}
$$
holds uniformly in $\t,\vp \in (0,\pi)$, $\t\ne\vp$.
\end{lem}

\begin{lem}\label{lem:new}
Let $\ab> -1$, $M,N \in \N$, $L,\gamma_1,\gamma_2 \in \{0,1\}$, $W \ge 1$ and $1 \le p \le \infty$
be fixed and such that $W/p+\gamma_1+\gamma_2-L-N-M\ge1$.
Then the estimate
$$
(\sin\theta)^{\gamma_1}(\sin\varphi)^{\gamma_2}\Big\|\partial_{\vp}^{L} \partial_{\t}^N \partial_t^M H_t^{\a+1, \b+1}(\t,\vp) \Big\|_{L^p(\mathbb{R}_+,t^{W-1} dt)}\lesssim
\frac{1}{|\t-\vp|} \; \frac{1}{\mu_{\ab}^+(B(\t,|\t-\vp|))}
$$
holds uniformly in $\t,\vp \in (0,\pi)$, $\t\ne\vp$.
\end{lem}

To prove Lemmas \ref{lem:gr} and \ref{lem:new} we need to invoke some auxiliary results.
The first of them describes how the measure of a ball changes when one increases
the parameters of $\mu_{\ab}^{+}$.
\begin{lem}[{see \cite[Lemma 4.2]{NS1}}] \label{cor:compare}
Let $\alpha,\beta>-1$. Given any $\xi\ge0$, we have
$$
\mu_{\alpha+\xi,\beta+\xi}^{+}(B(\theta,|\theta-\varphi|))\simeq(\theta+\varphi)^{2\xi}(\pi-\theta+\pi-\varphi)^{2\xi}\mu_{\alpha,\beta}^{+}(B(\theta,|\theta-\varphi|)),
$$
uniformly in $\theta,\varphi\in(0,\pi)$.
\end{lem}

The next necessary result provides some elementary bounds, their proof is a simple exercise. 
For aesthetic reasons, the constants appearing
on the right-hand sides are optimal, but for our purposes any constants would be suitable.
Obviously, the bound (a) holds with the roles of $\theta$ and $\varphi$ exchanged.

\begin{lem}\label{estimates}
For $\theta, \varphi \in (0,\pi)$ we have the estimates 
\begin{itemize}
\item[(a)]
$$
\frac{|\theta-\varphi|\varphi(\pi-\varphi)}{(\theta+\varphi)^{2}(\pi-\theta+\pi-\varphi)^{2}}
	\le \frac {1}{4\pi},
$$
\item[(b)]
$$
\frac{\theta\,\varphi \,(\pi-\theta)(\pi-\varphi)}{(\theta+\varphi)^{2}(\pi-\theta+\pi-\varphi)^{2}}
	\le \frac{1}{16},
$$
\item[(c)]
$$
\frac{|\theta-\varphi|}{(\theta+\varphi)(\pi-\theta+\pi-\varphi)}\le \frac{1}{\pi}.
$$
\end{itemize}
\end{lem}

The last auxiliary result is a special case of the generalization of \cite[Lemma 4.3]{NS1}
established in \cite{L1}.
Notice that this result is restricted to the range $\ab\geq-1/2$. 

\begin{lem}[{see \cite[Lemma 4.4]{L1}}]\label{lem:bridgeup}
Let $\alpha,\beta\ge-1/2$ and $\kappa \ge 0$.
Then
$$
\iint\frac{d\Pi_{\alpha+\kappa}(u)\,d\Pi_{\beta+\kappa}(v)}{\q^{\alpha+\beta+3/2+\kappa}}
 \lesssim\frac{1}{\mu_{\alpha,\beta}^{+}(B(\theta,|\theta-\varphi|))}
$$
uniformly in $\theta,\varphi\in(0,\pi)$, $\theta\neq\varphi$.
\end{lem}

We are now prepared to give proofs of Lemmas \ref{lem:gr} and \ref{lem:new}.

\begin{proof}[{Proof of Lemma \ref{lem:gr}}]
Let
\begin{align*}
A_0 & := \big\|\partial_{\vp}^{L} \partial_{\t}^N \partial_t^M H_t^{\a+1, \b+1}(\t,\vp) \big\|_{L^p((0,1),t^{W-1} dt)}, \\
A_{\infty} & := \big\|\partial_{\vp}^{L} \partial_{\t}^N \partial_t^M H_t^{\a+1, \b+1}(\t,\vp) \big\|_{L^p((1,\infty),t^{W-1} dt)}.
\end{align*}
Clearly, it is enough to estimate $(\sin\theta)^{\gamma_1} (\sin\varphi)^{\gamma_2}A_0$ and the analogous
expression with $A_{\infty}$ replacing $A_0$. By Lemma \ref{cor:tL} we have
\begin{equation*}
(\sin\theta)^{\gamma_1}(\sin\varphi)^{\gamma_2}A_\infty\lesssim 1
	\lesssim \frac{1}{\mu_{\ab}^+(B(\t,|\t-\vp|))},
\end{equation*}
thus it remains to estimate the expression related to $A_0$.

To begin with, assume that $p < \infty$. 
Using sequently Lemma \ref{cor:t1} (i), Minkowski's inequality and then the assumption
$W/p+\gamma_1+\gamma_2-L-N-M\ge 2$ together with the boundedness of $\q$, we get
\begin{align*}
A_0&\lesssim\iint\d\Pi_{\alpha+1}(u)\,d\Pi_{\beta+1}(v)\,\bigg(\int_{0}^{1}
	\frac{t^{W-1}\,dt}{(t^{2}+\q)^{p[\alpha+\beta+5/2+W/(2p)+(\gamma_1+\gamma_2)/2]}}\bigg)^{1/p}.
\end{align*}
Changing the variable of integration $t\mapsto\sqrt{\q}t$ and enlarging the upper limit of integration
in the resulting integral to infinity (such the integral converges), we see that
\begin{equation*}
A_0\lesssim \iint\frac{d\Pi_{\alpha+1}(u)\,d\Pi_{\beta+1}(v)}{{\q}^{\alpha+\beta+5/2+(\gamma_1+\gamma_2)/2}}.
\end{equation*}
The last estimate is valid also for $p=\infty$, as can be seen directly by using the assumption
$\gamma_1+\gamma_2-L-N-M \ge 2$ and the trivial bound
$$
\frac{1}{(t^{2}+\q)^{\a+\b+5/2+(\gamma_1+\gamma_2)/2}}\le  \frac{1}{\q^{\a+\b+5/2+(\gamma_1+\gamma_2)/2}}.
$$
To proceed, it is convenient to distinguish two cases. Combined together, they complete the proof.

\noindent\textbf{Case 1:} $\gamma_1=\gamma_2=1$. Applying sequently Lemma \ref{lem:bridgeup} 
(with $\kappa=0$), Lemma \ref{cor:compare} (with $\xi=1$) and Lemma \ref{estimates} (b) we get
\begin{align*}
(\sin\theta)^{\gamma_1}(\sin\varphi)^{\gamma_2}A_0&\lesssim\t(\pi-\t)\vp(\pi-\vp)\iint\frac{d\Pi_{\alpha+1}(u)\,d\Pi_{\beta+1}(v)}{{\q}^{(\a+1)+(\b+1)+3/2}}\\
&\lesssim\frac{\t(\pi-\t)\vp(\pi-\vp)}{\mu_{\a+1,\b+1}^+(B(\t,|\t-\vp|))} \\
& \simeq\frac{\t(\pi-\t)\vp(\pi-\vp)}{(\t+\vp)^2(\pi-\t+\pi-\vp)^2}\frac{1}{\mu_{\a,\b}^+(B(\t,|\t-\vp|))}\\
&\le\frac{1}{\mu_{\a,\b}^+(B(\t,|\t-\vp|))}.
\end{align*}

\noindent \textbf{Case 2:} $\gamma_1+\gamma_2=1$. By Lemma \ref{lem:bridgeup} (with $\kappa=1/2$)
and Lemma \ref{cor:compare} (with $\xi=1/2$)
\begin{align*}
(\sin\theta)^{\gamma_1}(\sin\varphi)^{\gamma_2}A_0&\lesssim[\t(\pi-\t)]^{\gamma_1}[\vp(\pi-\vp)]^{\gamma_2}\iint\frac{d\Pi_{(\a+1/2)+1/2}(u)\,d\Pi_{(\b+1/2)+1/2}(v)}{{\q}^{(\a+1/2)+(\b+1/2)+3/2+1/2}}\\
&\lesssim\frac{[\t(\pi-\t)]^{\gamma_1}[\vp(\pi-\vp)]^{\gamma_2}}{\mu_{\a+1/2,\b+1/2}^+(B(\t,|\t-\vp|))}\\
& \simeq\frac{[\t(\pi-\t)]^{\gamma_1}[\vp(\pi-\vp)]^{\gamma_2}}{(\t+\vp)(\pi-\t+\pi-\vp)}
	\frac{1}{\mu_{\a,\b}^+(B(\t,|\t-\vp|))}\\
&\le\frac{1}{\mu_{\a,\b}^+(B(\t,|\t-\vp|))}.
\end{align*}
\end{proof}

\begin{proof}[{Proof of Lemma \ref{lem:new}}]
Arguing in the same way as in the corresponding part of the proof of Lemma \ref{lem:gr} 
we are reduced to estimating $(\sin\theta)^{\gamma_1}(\sin\varphi)^{\gamma_2}A_0$ and then
arrive at the estimate
$$
A_0\lesssim\iint\frac{d\Pi_{\alpha+1}(u)\,d\Pi_{\beta+1}(v)}{{\q}^{\alpha+\beta+3+(\gamma_1+\gamma_2)/2}}.
$$
Now, similarly as before, we distinguish two cases which jointly complete the proof. 

\noindent\textbf{Case 1:} $\gamma_1=\gamma_2=0$. Using sequently Lemma \ref{lem:bridgeup} 
(with $\kappa=1/2$), Lemma \ref{cor:compare} (with $\xi=1/2$) and Lemma \ref{estimates} (c) we get
\begin{align*}
(\sin\theta)^{\gamma_1}(\sin\varphi)^{\gamma_2}A_0&\lesssim\iint\frac{d\Pi_{(\a+1/2)+1/2}(u)\,d\Pi_{(\b+1/2)+1/2}(v)}{{\q}^{(\a+1/2)+(\b+1/2)+3/2+1/2}}\\
&\lesssim\frac{1}{\mu_{\a+1/2,\b+1/2}^+(B(\t,|\t-\vp|))}\\
&\simeq\frac{|\t-\vp|}{(\t+\vp)(\pi-\t+\pi-\vp)}\frac{1}{|\t-\vp|\mu_{\a,\b}^+(B(\t,|\t-\vp|))}\\
&\le\frac{1}{|\t-\vp|\mu_{\a,\b}^+(B(\t,|\t-\vp|))}.
\end{align*}

\noindent\textbf{Case 2:} $\gamma_1+\gamma_2\ge1$. Due to the bound $\q \ge |\theta-\varphi|^2/\pi^2$ we have
\begin{align*}
(\sin\theta)^{\gamma_1}(\sin\varphi)^{\gamma_2}A_0&\lesssim(\sin\theta)^{\gamma_1}(\sin\varphi)^{\gamma_2}\iint\frac{d\Pi_{\a+1}(u)\,d\Pi_{\b+1}(v)}{{\q}^{\a+\b+3+(\gamma_1+\gamma_2)/2}}\\
&\lesssim \frac{1}{|\t-\vp|}(\sin\theta)^{\gamma_1}(\sin\varphi)^{\gamma_2}\iint\frac{d\Pi_{\a+1}(u)\,d\Pi_{\b+1}(v)}{{\q}^{\a+\b+5/2+(\gamma_1+\gamma_2)/2}}\\
&\lesssim\frac{1}{|\t-\vp|}\frac{1}{\mu_{\a,\b}^+(B(\t,|\t-\vp|))},
\end{align*}
where the last bound follows from the estimates obtained in the proof of Lemma \ref{lem:gr}.

The proof is finished.
\end{proof}

\section{Kernel estimates}\label{sec:ker}

Let $\mathbb{B}$ be a Banach space and let $K(\t,\vp)$ be a kernel defined on 
$(0,\pi)\times(0,\pi)\backslash \{ (\t,\vp):\t=\vp \}$ and taking values in $\mathbb{B}$.
We say that $K(\t,\vp)$ satisfies the so-called \emph{standard estimates}
in the sense of the space of homogeneous type
$((0,\pi), d\mu_{\ab}^+,|\cdot|)$ if it satisfies the growth estimate
\begin{equation} \label{gr}
\|K(\t,\vp)\|_{\mathbb{B}} \lesssim \frac{1}{\mu_{\ab}^+(B(\t,|\t-\vp|))}, \qquad \theta \neq \varphi,
\end{equation}
and the smoothness estimates
\begin{align}
\| K(\t,\vp)-K(\t',\vp)\|_{\mathbb{B}} & \lesssim \frac{|\t-\t'|}{|\t-\vp|}\,
 \frac{1}{\mu_{\ab}^+(B(\t,|\t-\vp|))},
\qquad |\t-\vp|>2|\t-\t'|, \label{sm1}\\
\| K(\t,\vp)-K(\t,\vp')\|_{\mathbb{B}} & \lesssim \frac{|\vp-\vp'|}{|\t-\vp|}\,
 \frac{1}{\mu_{\ab}^+(B(\t,|\t-\vp|))},
\qquad |\t-\vp|>2|\vp-\vp'| \label{sm2}.
\end{align}
Notice that in these formulas the interval 
$B(\t,|\t-\vp|)=(\theta-|\theta-\varphi|,\theta+|\theta-\varphi|)\cap (0,\pi)$ 
can be replaced by $B(\vp,|\t-\vp|)$, in view of the doubling property of $\mu_{\ab}^+$.

When $K(\t,\vp)$ is scalar-valued, i.e.\ $\mathbb{B}=\mathbb{C}$, it is well known that the bounds \eqref{sm1}
and \eqref{sm2} follow from the more convenient gradient estimate
\begin{equation} \label{grad}
{\|\partial_{\t} K(\t,\vp)\|}_{\mathbb{B}} + {\|\partial_{\vp} K(\t,\vp)\|}_{\mathbb{B}}
\lesssim \frac{1}{|\t-\vp|\mu_{\ab}^+(B(\t,|\t-\vp|))}, \qquad \theta \neq \varphi.
\end{equation}
The same holds also in the vector-valued cases we consider, see \cite[Section 4]{parameters}, 
the derivatives in \eqref{grad} being then taken in the weak sense. The latter means that for any functional
$\texttt{v}\in \mathbb{B}^*$ 
\begin{equation}\label{wder}
\texttt{v}\big( \partial_{\t} K(\t,\vp)\big)  =
\partial_{\t}  \texttt{v}\big( K(\t,\vp)\big)
\end{equation}
and similarly for $\partial_{\vp}$.
If these weak derivatives $\partial_{\t} K(\t,\vp)$ and $\partial_{\vp} K(\t,\vp)$ exist 
as elements of $\mathbb{B}$ and their norms satisfy \eqref{grad}, the scalar-valued case applies 
and \eqref{sm1} and \eqref{sm2} follow.

Now we are in a position to prove Theorem \ref{thm:stand}.
In what follows always $(\theta, \varphi)\in (0,\pi)\times(0,\pi)$ and $\theta\neq\varphi$.
Further, we tacitly assume that changing orders of certain differentiations and integrations is legitimate.
In fact, such manipulations can be easily justified with the aid of the estimates obtained along the proof
of Theorem \ref{thm:stand} and the dominated convergence theorem.

\subsection{{Proof of Theorem \ref{thm:stand}, the part related to Riesz transforms}} 
 
\begin{proof}[Proof of Theorem \ref{thm:stand}; the case of $\widetilde{R}_{N}^{\ab}(\theta,\varphi)$] 
 We consider two cases, depending on whether $N \ge 1$ is even or odd.
  
\noindent \textbf{Case 1:} $N$ is even.
Let $N=2k_{0}$ with $k_{0}\ge 1$. We first proceed as in \cite[p.\,281]{L1}.
Term by term differentiation of the series defining $\widetilde{H}_{t}^{\alpha,\beta}(\theta,\varphi)$ allows one to verify that this kernel satisfies in the strip $(t,\theta)\in(0,\infty)\times(0,\pi)$ the Laplace equation based on the modified Jacobi operator $\delta\delta_{\ab}^*+\lambda_0^{\ab}.$ This can be written as
 \begin{equation} \label{pochodne'}
\delta_{2}^{\textrm{odd}}\widetilde{H}_{t}^{\alpha,\beta}(\theta,\varphi)=\partial_{t}^{2}\widetilde{H}_{t}^{\alpha,\beta}(\theta,\varphi)-\lambda_0^{\ab}\widetilde{H}_{t}^{\alpha,\beta}(\theta,\varphi).
\end{equation}
  Iterating this relation we get
\begin{equation} \label{pochodne''}
\delta_{N}^{\textrm{odd}}\widetilde{H}_{t}^{\alpha,\beta}(\theta,\varphi)=\sum_{j=0}^{k_{0}}c_{j}\,\partial_{t}^{2(k_{0}-j)}\widetilde{H}_{t}^{\alpha,\beta}(\theta,\varphi),
\end{equation}
 with some constants $c_{j}$.
 Consequently, it suffices to show that for each $0 \le j \le k_0$ the kernel
 $$\widetilde{S}_{j}^{\alpha,\beta}(\theta,\varphi):=\int_{0}^{\infty}\partial_{t}^{2j}\widetilde{H}_{t}^{\alpha,\beta}(\theta,\varphi)\,t^{2k_{0}-1}\,dt$$
   satisfies conditions (\ref{gr}) and (\ref{grad}) with $\mathbb{B}=\mathbb{C}$. 

First we show the growth estimate \eqref{gr}. We have
\begin{align*}
|\widetilde{S}_{j}^{\alpha,\beta}(\theta,\varphi)|\le\sin\theta\,\sin\varphi\,\int_{0}^{\infty}|\partial_{t}^{2j}H_{t}^{\alpha+1,\beta+1}(\theta,\varphi)|\,t^{2k_{0}-1}dt.
\end{align*}
From here the growth bound follows from Lemma \ref{lem:gr} (applied with $L=N=0, M=2j, W=2k_0, \gamma_1=\gamma_2=1$, $p=1$).

We pass to the gradient estimate \eqref{grad}.
For symmetry reasons, it is enough to consider only the derivative in $\theta$. A simple computation shows that 
\begin{align*}
|\partial_{\theta}\widetilde{S}_{j}^{\alpha,\beta}(\theta,\varphi)|&\le
\sin\theta\,\sin\varphi\,\int_{0}^{\infty}|\partial_{\theta}\partial_{t}^{2j} H_{t}^{\alpha+1,\beta+1}(\theta,\varphi)|\,t^{2k_{0}-1}\,dt\\
&\quad+\sin\varphi\,\int_{0}^{\infty}|\partial_{t}^{2j}
	H_{t}^{\alpha+1,\beta+1}(\theta,\varphi)|\,t^{2k_{0}-1}\,dt.
\end{align*}
Then both the terms on the right-hand side are treated directly by Lemma \ref{lem:new} (applied with $L=0, N=1, M=2j, W=2k_0, \gamma_1=\gamma_2=1, p=1$ in case of the first term and $L=N=0, M=2j, W=2k_0, \gamma_1=0, \gamma_2=1, p=1$ in case of the second one). Thus the reasoning for the case of $N$ even is finished.

\noindent \textbf{Case 2:} $N$ is odd. Let now $N=2k_{0}+1$ for some $k_{0}\ge 0$.
In view of \eqref{pochodne''}, we have 
\begin{align*}
\widetilde{R}_{N}^{\alpha,\beta}(\theta,\varphi)&=\int_{0}^{\infty}\sum_{j=0}^{k_{0}}c_{j}\,\delta^{*}_{\ab}\partial_{t}^{2(k_{0}-j)}\widetilde{H}_{t}^{\alpha,\beta}(\theta,\varphi)\,t^{2k_{0}}\,dt.
\end{align*} 
Again, we observe that it is enough to show that, for each $0\le j\le k_{0}$, the kernel
$$\widetilde{\mathcal{S}}_{j}^{\alpha,\beta}(\theta,\varphi):=\int_{0}^{\infty}\delta^{*}_{\ab}\partial_{t}^{2j}\widetilde{H}_t^{\alpha,\beta}(\theta,\varphi)\,t^{2k_{0}}\,dt$$
satisfies the standard estimates (here and elsewhere $\delta_{\ab}^*$ acts always on $\theta$ variable).

First we show the growth bound. A direct computation reveals that
\begin{align*}
|\widetilde{\mathcal{S}}_{j}^{\alpha,\beta}(\theta,\varphi)|&\lesssim\sin\theta\,\sin\varphi\,\int_{0}^{\infty}|\partial_{\theta}\partial_{t}^{2j}\,H_{t}^{\alpha+1,\beta+1}(\theta,\varphi)|\,t^{2k_{0}}\,dt\\ 
&\quad+\sin\varphi\,\int_{0}^{\infty}|\partial_{t}^{2j}\,H_{t}^{\alpha+1,\beta+1}(\theta,\varphi)|\,t^{2k_{0}}\,dt.
\end{align*}
Both the above terms fall under the scope of Lemma \ref{lem:gr} (specified to $L=0, N=1, M=2j, W=2k_0+1, \gamma_1=\gamma_2=1,  p=1$ in case of the first summand and $L=N=0, M=2j, W=2k_0+1, \gamma_1=0, \gamma_2=1, p=1$ in case of the second one), hence \eqref{gr} follows.

For the gradient condition we use \eqref{pochodne'} to write
 
\begin{align*}
|\partial_{\theta}\widetilde{\mathcal{S}}_{j}^{\alpha,\beta}(\theta,\varphi)|&\lesssim\sin\theta\,\sin\varphi\,\int_{0}^{\infty}|\partial_{t}^{2j+2}\,H_{t}^{\alpha+1,\beta+1}(\theta,\varphi)|\,t^{2k_{0}}\,dt\\ 
&\quad+\sin\theta\sin\varphi\,\int_{0}^{\infty}|\partial_{t}^{2j}\,H_{t}^{\alpha+1,\beta+1}(\theta,\varphi)|\,t^{2k_{0}}\,dt.
\end{align*}
Now a double application of Lemma \ref{lem:new} (with $L=N=0, M=2j+2, W=2k_0+1, \gamma_1=\gamma_2=1, p=1$ in case of the first component and $L=N=0, M=2j, W=2k_0+1, \gamma_1=\gamma_2=1, p=1$ in case of the second one)  gives
$$
|\partial_{\theta}\widetilde{\mathcal{S}}_{j}^{\alpha,\beta}(\theta,\varphi)|\lesssim\frac{1}{|\theta-\varphi|}\,\frac{1}{\mu_{\alpha,\beta}^{+}(B(\theta,|\theta-\varphi|))}.
$$

It remains to estimate the derivative in $\varphi$. We have
\begin{align*}
|\partial_{\varphi}\widetilde{\mathcal{S}}_{j}^{\alpha,\beta}(\theta,\varphi)|&\lesssim\sin\theta\,\sin\varphi\,\int_{0}^{\infty}|\partial_{\varphi}\partial_{\theta}\partial_{t}^{2j}\,H_{t}^{\alpha+1,\beta+1}(\theta,\varphi)|\,t^{2k_{0}}\,dt\\ 
&\quad+\sin\varphi\,\int_{0}^{\infty}|\partial_{\varphi}\partial_{t}^{2j}\,H_{t}^{\alpha+1,\beta+1}(\theta,\varphi)|\,t^{2k_{0}}\,dt\\ 
&\quad+\sin\theta\,\int_{0}^{\infty}|\partial_{\theta}\partial_{t}^{2j}\,H_{t}^{\alpha+1,\beta+1}(\theta,\varphi)|\,t^{2k_{0}}\,dt\\
&\quad+\int_{0}^{\infty}|\partial_{t}^{2j}\,H_{t}^{\alpha+1,\beta+1}(\theta,\varphi)|\,t^{2k_{0}}\,dt.
\end{align*}
All the terms on the right-hand side can be treated by means of Lemma \ref{lem:new} (taken, respectively, with $L=N=1,M=2j, W=2k_0+1, \gamma_1=\gamma_2=1, p=1$; $L=1, N=0, M=2j, W=2k_0+1, \gamma_1=0, \gamma_2=1, p=1$; $L=0, N=1, M=2j, W=2k_0+1, \gamma_1=1, \gamma_2=0, p=1$ and $L=N=0, M=2j, W=2k_0+1, \gamma_1=\gamma_2=0, p=1$). The smoothness bound follows. 
\end{proof}

\subsection{{Proof of Theorem \ref{thm:stand}, the part related to square functions}}

\begin{proof}[Proof of Theorem \ref{thm:stand}; the case of $\{\partial_t^M\delta_N^{\textrm{\emph{even}}} H_t^{\ab}(\theta,\varphi)\}_{t>0}$.] We consider two cases.

\noindent \textbf{Case 1:} $N$ is even.
Let $k_{0}=N/2\ge 0$. Term by term differentiation of the defining series reveals that $H_{t}^{\alpha,\beta}(\theta,\varphi)$ satisfies in the strip $(t,\theta)\in(0,\infty)\times(0,\pi)$ the Laplace equation based on the Jacobi Laplacian, which can be written as (see \cite[p.\,278]{L1})
\begin{equation*}
\delta_2^{\textrm{even}}H_{t}^{\alpha,\beta}(\theta,\varphi)=\partial_{t}^{2}H_{t}^{\alpha,\beta}(\theta,\varphi)-\lambda_0^{\ab}H_{t}^{\alpha,\beta}(\theta,\varphi).
\end{equation*}
 Iterating this identity, we get 
 \begin{align}\label{pochodne}
 \de H_{t}^{\alpha,\beta}(\theta,\varphi)=\sum_{j=0}^{k_{0}}c_{j}\,\partial_{t}^{2(k_{0}-j)}H_{t}^{\alpha,\beta}(\theta,\varphi) 
 \end{align}
 with some constants $c_{j}$.
 Consequently, it is enough to verify that for each $0 \le j \le k_0$ the vector-valued kernel 
$$
T_{j}^{\alpha,\beta}(\theta,\varphi):=\big\{\partial_{t}^{M+2j}H_{t}^{\alpha,\beta}(\theta,\varphi)\big\}_{t>0}
$$
satisfies the standard estimates \eqref{gr} and \eqref{grad} with $\mathbb{B}=L^2(\R_+,t^{2M+2N-1}dt)$.

The growth condition \eqref{gr} for $T_{j}^{\alpha,\beta}(\theta,\varphi)$ follows directly from Lemma \ref{cor:t1} (with $L=N=0$) combined with the boundedness of $\q$, Lemma \ref{lem:finbridge} (taken with $W=2N+2M, s=0, p=2$) and Lemma \ref{cor:tL} (specified to $L=N=0, W=2N+2M, p=2$). 
More precisely, here Lemma \ref{cor:tL} cannot be applied directly when $M=j=0$ and $\a+\b= -1$.
However,
in the singular case $\a+\b=-1$ the decomposition \eqref{pochodne} reduces to
$\de H_{t}^{\alpha,\beta}(\theta,\varphi) = \partial_t^N H_t^{\ab}(\theta,\varphi)$ and so we need to
estimate only the kernel $T_{N/2}^{\ab}(\theta,\varphi)$. In this situation Lemma \ref{cor:tL} again
applies (recall that $M+N>0$) and the growth bound follows.

To prove the smoothness condition, because of the symmetry, it suffices to consider the derivative in $\theta$.  It is not hard to check that the weak derivative 
$\partial_{\theta}T_{j}^{\alpha,\beta}(\theta,\varphi)$ in the sense of \eqref{wder} equals $\big\{\partial_{\t}\partial_{t}^{M+2j}H_{t}^{\alpha,\beta}(\theta,\varphi)\big\}_{t>0}$, 
cf.\ \cite[Proof of Theorem 4.1]{parameters}.
Therefore we need to show that
$$
\big\|\partial_{\t}\partial_{t}^{M+2j}H_{t}^{\alpha,\beta}(\theta,\varphi)\big\|_{L^2(\R_+,t^{2M+2N-1}dt)}\lesssim\frac{1}{|\theta-\varphi|}\frac{1}{\m_{\ab}^{+}(B(\theta,|\theta-\varphi|))}.$$
This, however, follows again from Lemma \ref{cor:t1} (specified to $L=0$ and $N=1$), the boundedness
of $\q$ combined with Lemma \ref{lem:finbridge} (applied with $W=2M+2N, s=1, p=2$) and Lemma \ref{cor:tL} (with  $L=0, N=1, W=2M+2N, p=2$).

\noindent \textbf{Case 2:} $N$ is odd, say $N=2k_{0}+1$. We take into account \eqref{pochodne} and observe that it suffices, for each $0\le j\le k_{0}$, to verify the standard estimates for the vector-valued kernel
$$\mathcal{T}_{j}^{\alpha,\beta}(\theta,\varphi):=\big\{\partial_{\theta}\partial_{t}^{M+2j}H_{t}^{\alpha,\beta}(\theta,\varphi)\big\}_{t>0}.$$
This, however, is done by a straightforward repetition of the arguments used for the case of $N$ even. 
We omit the details.
\end{proof}

\begin{proof}[Proof of Theorem \ref{thm:stand}; the case of $\{\partial_t^M\delta_N^{\textrm{\emph{odd}}}\widetilde{H}_t^{\ab}(\theta,\varphi)\}_{t>0}$.]
Again, we consider two cases.

\noindent \textbf{Case 1:} $N$ is even.
Let $N=2k_{0}$ with $k_0\ge0$. In view of \eqref{pochodne''}, it is enough to show that for each $0\le j\le k_{0}$ the vector-valued kernel
$$\widetilde{T}_{j}^{\alpha,\beta}(\theta,\varphi):=\big\{\partial_{t}^{M+2j}\widetilde{H}_{t}^{\alpha,\beta}(\theta,\varphi)\big\}_{t>0}$$
satisfies the standard estimates with $\mathbb{B}=L^2(\R_+,t^{2M+2N-1}dt)$.

The growth bound \eqref{gr} is straightforward, since by Lemma \ref{lem:gr} (specified to  $L=N=0, W=2M+2N, \gamma_1=\gamma_2=1, p=2$) we have
\begin{align*}
\big\|\widetilde{T}_{j}^{\alpha,\beta}(\theta,\varphi)\big\|_{L^2(\R_+,t^{2M+2N-1}dt)}&= 
\frac1{4}\, {\sin\theta\,\sin\varphi}\,\big\|\partial_{t}^{M+2j}H_{t}^{\alpha+1,\beta+1}(\theta,\varphi)\big\|_{L^2(\R_+,t^{2M+2N-1}dt)}\\
&\lesssim\frac{1}{\mu_{\a,\b}^{+}(B(\theta,|\theta-\varphi|))}.
\end{align*}  

We pass to the smoothness condition \eqref{grad}. For symmetry reasons, it is enough to treat the derivative in $\theta$. One easily verifies that the weak derivative $\partial_{\theta}\widetilde{T}_{j}^{\alpha,\beta}(\theta,\varphi)$ in the sense of \eqref{wder} is given by $\big\{\partial_{\t}\partial_{t}^{M+2j}\widetilde{H}_{t}^{\alpha,\beta}(\theta,\varphi)\big\}_{t>0}$. Then we have
\begin{align*}
\big\|\partial_{\t}\partial_{t}^{M+2j}\widetilde{H}_{t}^{\alpha,\beta}(\theta,\varphi)\big\|_{L^2(\R_+,t^{2M+2N-1}dt)}&\le\sin\theta\sin\varphi\,\big\|\partial_{\theta}\,\partial_{t}^{M+2j}H_{t}^{\alpha+1,\beta+1}(\theta,\varphi)\big\|_{L^2(\R_+,t^{2M+2N-1}dt)}\\
&\quad +\sin\varphi\,\big\|\partial_{t}^{M+2j}H_{t}^{\alpha+1,\beta+1}(\theta,\varphi)\big\|_{L^2(\R_+,t^{2M+2N-1}dt)}.
\end{align*}
Each of the above terms can be estimated suitably by means of Lemma \ref{lem:new} (applied with $L=0, N=1, W=2M+2N, \gamma_1=\gamma_2=1, p=2$ in case of the first summand and $L=N=0, W=2M+2N, \gamma_1=0, \gamma_2=1, p=2$ in case of the second one). The conclusion follows. 

\noindent \textbf{Case 2:} $N$ is odd.
Let $N=2k_{0}+1$ with $k_0\ge0$. The kernel we need to estimate is, see \eqref{pochodne''},
\begin{equation*}
\big\{\partial_{t}^{M}\ko\widetilde{H}_{t}^{\alpha,\beta}(\theta,\varphi)\big\}_{t>0}=\bigg\{\sum_{j=0}^{k_{0}}c_{j}\,\delta^{*}_{\ab}\partial_{t}^{M}\partial_{t}^{2(k_{0}-j)}\widetilde{H}_{t}^{\alpha,\beta}(\theta,\varphi)\bigg\}_{t>0}.
\end{equation*}
Thus it is enough to show that for each $0\le j\le k_{0}$ the kernel
$$\widetilde{\mathcal{T}}_{j}^{\alpha,\beta}(\theta,\varphi):=\big\{\delta^{*}_{\ab}\partial_{t}^{M+2j}\widetilde{H}_{t}^{\alpha,\beta}(\theta,\varphi)\big\}_{t>0}
$$ satisfies the standard estimates.

A direct computation reveals that
\begin{align*}
\big\|\widetilde{\mathcal{T}}_{j}^{\alpha,\beta}(\theta,\varphi)\big\|_{L^2(\R_+,t^{2M+2N-1}dt)}&\lesssim\sin\theta\,\sin\varphi\,\big\|\partial_{\theta}\partial_{t}^{M+2j}H_{t}^{\alpha+1,\beta+1}(\theta,\varphi)\big\|_{L^2(\R_+,t^{2M+2N-1}dt)}\\
&\quad +\sin\varphi\,\big\|\partial_{t}^{M+2j}H_{t}^{\alpha+1,\beta+1}(\theta,\varphi)\big\|_{L^2(\R_+,t^{2M+2N-1}dt)}.
\end{align*}
Then the growth bound \eqref{gr} is a consequence of Lemma \ref{lem:gr} applied twice (with $L=0, N=1, W=2M+2N, \gamma_1=\gamma_2=1, p=2$ in case of the first term and with $L=N=0, W=2M+2N, \gamma_1=0, \gamma_2=1, p=2$ in case of the second one).

To show the smoothness condition \eqref{grad} we begin with the derivative in $\theta$. A simple computation involving \eqref{pochodne'} reveals that
\begin{align*}
\big\|\partial_\theta\widetilde{\mathcal{T}}_{j}^{\alpha,\beta}(\theta,\varphi)\big\|_{L^2(\R_+,t^{2M+2N-1}dt)}&\lesssim \sin\theta\sin\varphi\,\big\|\partial_{t}^{M+2j+2}H_{t}^{\alpha+1,\beta+1}(\theta,\varphi)\big\|_{L^2(\R_+,t^{2M+2N-1}dt)}\\
&\quad+\sin\theta\sin\varphi\,\big\|\partial_{t}^{M+2j}H_{t}^{\alpha+1,\beta+1}(\theta,\varphi)\big\|_{L^2(\R_+,t^{2M+2N-1}dt)}.
\end{align*}
Both of the above terms are estimated with the aid of Lemma \ref{lem:new} (applied with $L=N=0, W=2M+2N, \gamma_1=\gamma_2=1, p=2$).

It remains to consider the derivative in $\vp$.
A straightforward computation leads to the bound
\begin{align*}
\big\|\partial_\vp\widetilde{\mathcal{T}}_{j}^{\alpha,\beta}(\theta,\varphi)\big\|_{L^2(\R_+,t^{2M+2N-1}dt)}&\lesssim \sin\theta\sin\varphi\,\big\|\partial_{\vp}\partial_{\t}\partial_{t}^{M+2j}H_{t}^{\alpha+1,\beta+1}(\theta,\varphi)\big\|_{L^2(\R_+,t^{2M+2N-1}dt)}\\
&\quad+\sin\vp\,\big\|\partial_{\vp}\partial_{t}^{M+2j}H_{t}^{\alpha+1,\beta+1}(\theta,\varphi)\big\|_{L^2(\R_+,t^{2M+2N-1}dt)}\\
&\quad+\sin\t\,\big\|\partial_{\t}\partial_{t}^{M+2j}H_{t}^{\alpha+1,\beta+1}(\theta,\varphi)\big\|_{L^2(\R_+,t^{2M+2N-1}dt)}\\
&\quad+\big\|\partial_{t}^{M+2j}H_{t}^{\alpha+1,\beta+1}(\theta,\varphi)\big\|_{L^2(\R_+,t^{2M+2N-1}dt)}.
\end{align*}
All of the above terms are controlled by the right-hand side of \eqref{grad}, which follows by applying repeatedly Lemma \ref{lem:new} (with $L=N=1, W=2M+2N, \gamma_1=\gamma_2=1, p=2$ in case of the first summand; with $L=1, N=0, W=2M+2N, \gamma_1=0, \gamma_2=1, p=2$ in case of the second one; with $L=0, N=1, W=2M+2N, \gamma_1=1, \gamma_2=0, p=2$ in case of the third one and with $L=N=0, W=2M+2N, \gamma_1=\gamma_2=0, p=2$ in case of the last one).
This finishes the reasoning.
\end{proof}

\subsection{{Proof of Theorem \ref{thm:stand}, the part related to spectral multipliers}}

\begin{proof}[Proof of Theorem \ref{thm:stand}; the case of $\widetilde{M}_{\phi}^{\alpha,\beta}(\theta,\varphi)$]

Since $\phi$ is bounded, we have
\begin{align*}
\widetilde{M}_{\phi}^{\alpha,\beta}(\theta,\varphi)\lesssim\sin\theta\sin\varphi\,\int_{0}^{\infty}|\partial_{t}H_{t}^{\alpha+1,\beta+1}(\theta,\varphi)|\,dt.
\end{align*}
Now to get the growth bound \eqref{gr} with $\mathbb{B}=\mathbb{C}$ it is enough to apply Lemma \ref{lem:gr} (specified to $L=N=0, M=1, W=1, \gamma_1=\gamma_2=1, p=1$).

To show the smoothness estimate \eqref{grad} we can restrict, for symmetry reasons, to the derivative in 
$\theta$. Taking into account the boundedness of $\phi$, we get
\begin{align*}
|\partial_{\theta}\widetilde{M}_{\phi}^{\alpha,\beta}(\theta,\varphi)|&\lesssim\sin\theta\sin\varphi\,\int_{0}^{\infty}|\partial_{\theta}\partial_{t}H_{t}^{\alpha+1,\beta+1}(\theta,\varphi)|\,dt\\
&\quad+\sin\varphi\,\int_{0}^{\infty}|\partial_{t}H_{t}^{\alpha+1,\beta+1}(\theta,\varphi)|\,dt.
\end{align*}
Both the terms on the right-hand side satisfy the desired estimate by Lemma \ref{lem:new} (specified to $L=0, N=M=1, W=1, \gamma_1=\gamma_2=1, p=1$ in case of the first summand and $L=N=0, M=1, W=1, \gamma_1=0, \gamma_2=1, p=1$ in case of the second one).
\end{proof}

\begin{proof}[Proof of Theorem \ref{thm:stand}; the case of $\widetilde{M}_{\nu}^{\alpha,\beta}(\theta,\varphi)$]
Taking into account the condition \eqref{L-Sc}
it is clear that proving the growth estimate for $\widetilde{M}_{\nu}^{\alpha,\beta}(\theta,\varphi)$ reduces to showing that 
\begin{align*} 
\sin\theta\sin\varphi\,\|  H_t^{\a+1,\b+1}(\t,\vp) \|_{L^\infty((0,1), dt)}
&\lesssim
\frac{1}{\mu_{\ab}^+(B(\t,|\t-\vp|))},\\
\sin\theta\sin\varphi\,\big\| e^{t\left| \frac{\a+\b+1}{2} \right|} H_t^{\a+1, \b+1}(\t,\vp)  \big\|_{L^\infty((1,\infty), dt)}
&\lesssim
\frac{1}{\mu_{\ab}^+(B(\t,|\t-\vp|))}.
\end{align*}
The first bound here follows by Lemma \ref{lem:gr} 
(applied with $L=N=M=0,  W=1, \gamma_1=\gamma_2=1, p=\infty$).
The second one is a consequence of Lemma \ref{cor:tLStieltjes} (taken with $L=N=0$);
when applying Lemma \ref{cor:tLStieltjes} here and below we use implicitly the fact that
$|\frac{\a+\b+1}{2}| < |\frac{(\a+1)+(\b+1)+1}{2}|$.

For symmetry reasons, verification of the gradient bound \eqref{grad} amounts to checking that
\begin{align*} 
B_0 & := \big\| \partial_{\theta}\widetilde{H}_t^{\ab}(\t,\vp)  \big\|_{L^\infty((0,1), dt)}
\lesssim\frac{1}{|\theta-\varphi|}\frac{1}{\m_{\ab}^{+}(B(\theta,|\theta-\varphi|))},\\
B_{\infty} & :=\big\| e^{t\left| \frac{\a+\b+1}{2} \right|} \partial_{\theta}\widetilde{H}_t^{\ab}(\t,\vp)  \big\|_{L^\infty((1,\infty), dt)}
\lesssim\frac{1}{|\theta-\varphi|}\frac{1}{\m_{\ab}^{+}(B(\theta,|\theta-\varphi|))}.
\end{align*}
An easy calculation reveals that
$$
\big|\partial_{\theta}\widetilde{H}_t^{\ab}(\t,\vp)\big|\leq \sin\theta\sin\varphi\,\big|\partial_{\theta}H_t^{\a+1,\b+1}(\t,\vp)\big|+\sin\vp\, H_t^{\a+1,\b+1}(\t,\vp).
$$
Therefore
$$
B_0 
\le\sin\theta\sin\varphi\,\big\| \partial_{\theta}H_t^{\a+1,\b+1}(\t,\vp)  \big\|_{L^\infty((0,1), dt)}
+\sin\varphi\,\big\| H_t^{\a+1,\b+1}(\t,\vp)  \big\|_{L^\infty((0,1), dt)}
$$
and each of the above terms can be estimated by means of Lemma \ref{lem:new} (applied with $L=0, N=1, M=0, W=1, \gamma_1=\gamma_2=1, p=\infty$ in case of the first summand and with $L=N=M=0, W=1, \gamma_1=0, \gamma_2=1, p=\infty$ in case of the second one). Thus the smoothness bound for $B_0$ follows.
Considering $B_{\infty}$, with the aid of Lemma \ref{cor:tLStieltjes} (applied twice: with $L=0$, $N=1$ and with $L=N=0$) we infer that
\begin{align*}
B_{\infty}
&\leq \big\| e^{t\left| \frac{\a+\b+1}{2} \right|} \partial_{\theta} H_t^{\a+1,\b+1}(\t,\vp)  \big\|_{L^\infty((1,\infty), dt)}
+\big\| e^{t\left| \frac{\a+\b+1}{2} \right|}  H_t^{\a+1,\b+1}(\t,\vp)  \big\|_{L^\infty((1,\infty), dt)}\\
&\lesssim 1\lesssim\frac{1}{|\theta-\varphi|}\frac{1}{\m_{\ab}^{+}(B(\theta,|\theta-\varphi|))}.
\end{align*}
\end{proof}

The proof of Theorem \ref{thm:stand} is finished.


\end{document}